\documentclass[12pt,reqno]{amsart}

\usepackage{amsfonts,amsmath,amsthm}
\usepackage{amssymb,epsfig}
\usepackage{enumerate} 

\usepackage[utf8]{inputenc}
\usepackage[euler-digits]{eulervm}
\usepackage{eucal}
\usepackage[unicode=true]{hyperref}
\hypersetup{colorlinks = true,pdftex}
\hypersetup{
     colorlinks,
     linkcolor={black!50!black},
     citecolor={blue}
}

\usepackage{pdfsync}

\usepackage{geometry}
\usepackage{graphicx}
\usepackage{epsfig}
\usepackage{tikz}
\usepackage{caption}
\usepackage{color} %color
\definecolor{vert}{rgb}{0,0.6,0}

\usepackage{comment}
\numberwithin{figure}{section}

\theoremstyle{plain}
\newtheorem{thm}{Theorem}[section]

\newtheorem{defn}{Definition}
\newtheorem{quest}{Question}

\newtheorem{ex}{Example}
\newtheorem{lem}[thm]{Lemma}
\newtheorem{cor}[thm]{Corollary}
\newtheorem{prop}[thm]{Proposition}
\theoremstyle{remark}
\newtheorem{rem}{\bf{Remark}}
\numberwithin{equation}{section}

\newcommand{\N}{\mathbb{N}}

\newcommand{\R}{\mathbb{R}}

\newcommand{\BUC}{{\rm BUC\,}}

\newcommand{\Lip}{{\rm Lip\,}}

\newcommand{\al}{\alpha}

\newcommand{\ol}{\overline}

\begin{document}

\title[Rate of convergence]
{State-constraint static Hamilton-Jacobi equations in nested domains}

\thanks{The authors were supported in part by NSF grant DMS-1664424.
The second and the third authors are supported in part by NSF CAREER grant DMS-1843320.}

\begin{abstract}
We study state-constraint static Hamilton--Jacobi equations in a  sequence of domains $\{\Omega_k\}_{k \in \N}$ in $\R^n$ such that $\Omega_k \subset \Omega_{k+1}$ for all $k\in \N$.
We obtain rates of convergence of $u_k$, the solution to the state-constraint problem in $\Omega_k$, to $u$, the solution to the corresponding problem in $\Omega =  \bigcup_{k \in \N} \Omega_k$.
In many cases,  the rates obtained are proven to be optimal.
Various new examples and discussions are provided at the end of the paper.
\end{abstract}

\author{Yeoneung Kim}
\address[Y. Kim]
{
Department of Mathematics, 
University of Wisconsin Madison, 480 Lincoln  Drive, Madison, WI 53706, USA}
\email{yeonkim@math.wisc.edu}

\author{Hung V. Tran}
\address[H. V. Tran]
{
Department of Mathematics, 
University of Wisconsin Madison, 480 Lincoln  Drive, Madison, WI 53706, USA}
\email{hung@math.wisc.edu}

\author{Son N. T. Tu}
\address[S. N.T. Tu]
{
Department of Mathematics, 
University of Wisconsin Madison, 480 Lincoln  Drive, Madison, WI 53706, USA}
\email{thaison@math.wisc.edu}

%\date{\today}

\maketitle

%\tableofcontents

%%%%%%%%%%%%%%%%%%%%%%%%%%%%%%%%%%%%%%%%%%%%%%%%%%%%%%%%%%%%%%%

\section{Introduction}\label{sec:intro}
Let $\{\Omega_k\}_{k \in \N}$ be a sequence of domains in $\R^n$ such that $\Omega_k \subset \Omega_{k+1}$ for all $k\in \N$.
We say that $\{\Omega_k\}_{k \in \N}$ is a sequence of nested domains.
Then, $\Omega = \bigcup_{k \in \N} \Omega_k$ is also a domain in $\R^n$.
Let $H:\ol{\Omega} \times \mathbb{R}^n\rightarrow \mathbb{R}$ be a given continuous Hamiltonian. In this paper, we are interested in studying state-constraint solutions to the following static Hamilton-Jacobi equations:
\begin{equation}\label{HJ-static-k}
u_k(x) + H\big(x,Du_k(x)\big) = 0 \qquad\text{in}\; \Omega_k, \tag{HJ$_k$}
\end{equation}
and
\begin{equation}\label{HJ-static}
u(x) + H\big(x,Du(x)\big) = 0 \qquad\text{in}\; \Omega \tag{HJ}.
\end{equation}
The precise definition of state-constraint viscosity solutions is given in Section \ref{sec:prelim}. Under some appropriate conditions, \eqref{HJ-static-k} has a unique  state-constraint viscosity solution $u_k \in \mathrm{C}(\ol \Omega_k)$ for each $k \in \N$,
and  \eqref{HJ-static} has a unique  state-constraint viscosity solution $u \in \mathrm{C}(\ol \Omega)$.
Furthermore, by a priori estimates and the stability results of viscosity solutions, we have that $u_k \to u$ locally uniformly on $\overline{\Omega}$.
Our main focus here is to study how fast this convergence is in two different types of nested domains.
%cases: the case that $\Omega$ is unbounded and the case that $\Omega$ is bounded.

\subsection{Assumptions} 
In the paper, we consider the following two prototypes of nested domains, which are
\begin{itemize}
\item[(P1)] $\Omega_k = B(0,k)$ and $\Omega = \bigcup_{ k \in \N} B(0,k) = \R^n$,
\item[(P2)] $\Omega_k = B(0,1-\frac{1}{k})$, and $\Omega = B(0,1)$. 
\end{itemize}
We list the main assumptions that will be used throughout the paper. 
\begin{itemize}
\item[(H1)] There exists $C_1>0$ such that
\begin{equation}\label{H0}
H(x,p) \geq -C_1 \qquad \text{ for all } (x,p)\in \overline{\Omega}\times \mathbb{R}^n. \tag{H1}
\end{equation}

\item[(H2)] There exists $C_2>0$ such that
\begin{equation}\label{H1}
\sup_{x\in \overline{\Omega}} |H(x,0)| \leq C_2. \tag{H2}
\end{equation}
\item[(H3a)] There exists a modulus $\omega_H:[0,\infty) \rightarrow [0,\infty)$, which is a nondecreasing function satisfying $\omega_H(0^+) = 0$ and
\begin{equation}\label{H3a}
\begin{cases}
|H(x,p) - H(y,p)| \leq \omega_H\big(|x-y|(1+|p|)\big),\\
|H(x,p) - H(x,q)| \leq \omega_H(|p-q|),
\end{cases} \tag{H3a}
\end{equation}
for $x,y \in \overline{\Omega}$ and $p,q\in \mathbb{R}^n$.

\item[(H3b)] For every $R>0$, there exists a modulus $\omega_R:[0,+\infty) \rightarrow [0,+\infty)$, which is nondecreasing with $\omega_R(0^+) = 0$ and
\begin{equation}\label{H3b}
\begin{cases}
|H(x,p) - H(y,p)| \leq \omega_R\big(|x-y|),\\
|H(x,p) - H(x,q)| \leq \omega_R(|p-q|),
\end{cases} \tag{H3b}
\end{equation}
for $x,y \in \overline{\Omega}$ and $p,q\in \mathbb{R}^n$ with $|p|,|q|\leq R$.
\item[(H3c)] For each $R>0$ there exists a constant $C_R$ such that
\begin{equation}\label{H3c}
\begin{cases}
|H(x,p) - H(y,p)| \leq C_R|x-y|,\\
|H(x,p) - H(x,q)| \leq C_R|p-q|, 
\end{cases} \tag{H3c}
\end{equation}
for $x,y\in \overline{\Omega}$ and $p,q \in \mathbb{R}^n$ with $|p|,|q|\leq R$.

\item[(H4)] $H$ satisfies the coercivity assumption
\begin{equation}\label{H4}
\lim_{|p|\rightarrow \infty} \left(\inf_{x\in \overline{\Omega}} H(x,p)\right) = +\infty. \tag{H4}
\end{equation}
\item[(H5)] $p\mapsto H(x,p)$ is convex for each $x\in\overline{\Omega}$.
\end{itemize}
Let us give some quick comments on the assumptions here.
Assumption \eqref{H0} is necessary to ask a meaningful question about the rate of convergence of $u_k$ to $u$.
See the discussion in Section \ref{Sec:discussion} in case where \eqref{H0} fails to hold.
Besides, it is clear that \eqref{H3b} is weaker than both \eqref{H3a} and \eqref{H3c}.

\subsection{Main results}
There have been many works in the literature on the well-posedness of state-constraint Hamilton-Jacobi equations and fully nonlinear elliptic equations.
The state-constraint problem for first-order convex Hamilton-Jacobi equations using optimal control frameworks was first studied in \cite{Soner1986, Soner1986a}.
The general nonconvex, coercive first-order equations was then discussed in \cite{Capuzzo-Dolcetta1990}.
For further developments in using optimal control formulation and obtaining optimal paths, we refer the readers to \cite{ Ishii1996,  Kim2000, ClarkeFrancisH.2002, CarlierGuillaume2010, BokanowskiOlivier2011, AltaroviciAlbert2013,Frankowska2013, Sedrakyan2016, WilliamH.Sandholm2018} for the finite dimensional cases, and \cite{Cannarsa1991, Kocan1998} for the infinite dimensional cases. 
See \cite{Camilli1996} for discrete numerical schemes, and \cite{Mitake2008} for large time behavior results.
We also refer to the classical books \cite{Barles1994, Bardi1997} and the references therein.

The state-constraint problem for second-order equations was first studied in \cite{Lasry1989} for the Laplacian, and in \cite{Armstrong2015a} for the general possibly degenerate diffusion matrices.
Boundary behavior of blow-up solutions was discussed  in \cite{Lasry1989, Leonori2008, Armstrong2015a}. 
Convex solutions with state-constraint boundary were constructed in \cite{Alvarez1997, Feldman2007}.
The convergence of solutions to the vanishing discount problems was proved in \cite{Ishii2017a}.

In terms of state-constraint problems in nested domains, up to our knowledge, there are only qualitative results in the literature in \cite{Capuzzo-Dolcetta1990, Armstrong2015a} where certain approximations were needed for the analysis of solutions. 
We provide here some first quantitative results on the rate of convergence of the solutions to \eqref{HJ-static-k} as $k$ goes to infinity in two different types ((P1) or (P2)) of nested domains.

\medskip

First of all, we show that the rate of convergence is $\mathcal{O}\left(\frac{1}{k^2}\right)$ for the prototype (P1) for general nonconvex Hamiltonians.

\begin{thm}\label{thm:approx1/k^2} 
Under the assumptions $\mathrm{(P1)}$, \eqref{H0}, \eqref{H1}, \eqref{H3c}, and \eqref{H4}, we have
\begin{itemize}
\item[(i)] $u(x) \leq u_k(x)$ for every $x\in \overline{B(0,k)}$,
\item[(ii)] there exists a constant $C>0$ depending only on $H$ such that
\begin{equation*}
0\leq u_k(x) - u(x) \leq \frac{C(1+|x|^2)}{k^{2}}
\end{equation*}
for $k \in \N$ and $x\in \overline{B(0,k)}$.
\end{itemize}
In particular, for any fixed $R>0$ and $|x|\leq R$,
\[
0\leq u_k(x) - u(x) \leq \frac{C(1+R^2)}{k^{2}}.
\]
\end{thm}

The condition that $|x|\leq R$ is important since there are examples where the estimate above fails at the boundary of $\Omega_k$. 
In Proposition \ref{example1}, we  have, for each $k \in \N$,  $|u_k(x) - u(x)| =1$ for some $x\in \partial \Omega_k$.

\begin{thm}\label{npo} 
Assume $\mathrm{(P1)}$.
Assume further that $H(x,p)=a(x)K(p)$ for $(x,p)\in \mathbb{R}^n \times \mathbb{R}^n$.
Here, $K(0)=0$, $K$ is locally Lipschitz and coercive in $\R^n$, and $a \in \mathrm{BUC}(\mathbb{R}^n)$ satisfies $\alpha \leq a(\cdot)\leq \beta$  for some given $\alpha,\beta>0$. 
Then, $u \equiv 0$, and for every $x\in \overline{B(0,k)}$, we have
\begin{equation*}
0\leq u_k(x) \leq \left(Ce^{\frac{|x|}{C}}\right)e^{-\frac{k}{C}},
\end{equation*}
where $C$ is a constant depending only on $H$.
In particular, for any fixed $R>0$, we have
\begin{equation*}
0 \leq u_k(x)\leq \left(C e^{\frac{R}{C}}\right) e^{-\frac{k}{C}}
\end{equation*}
for every $x\in \overline{B(0,R)}$ and $k\geq R$. 
In addition to that, this exponential rate is optimal.
\end{thm}

It is quite interesting to observe that we obtain the exponential rate of convergence for this particular class of nonconvex Hamiltonians and the rate is indeed optimal. When $a(x)$ is a positive constant, the assumption $K(0)=0$ in the theorem above can be removed. 

\begin{cor}\label{cor1} 
Assume $\mathrm{(P1)}$.
Assume further that $H(x,p)=H(p)$ for $(x,p)\in \mathbb{R}^n \times \mathbb{R}^n$.
Here, $H$ is locally Lipschitz and coercive in $\R^n$.
Then, $u \equiv - H(0)$, and for every $x\in \overline{B(0,k)}$, we have
\begin{equation*}
0\leq u_k(x) - u(x) \leq \left(Ce^{\frac{|x|}{C}}\right)e^{-\frac{k}{C}},
\end{equation*}
where $C$ is a constant depending only on $H$.
In particular, for any fixed $R>0$, we have
\begin{equation*}
0 \leq u_k(x) - u(x)\leq \left(C e^{\frac{R}{C}}\right) e^{-\frac{k}{C}}
\end{equation*}
for every $x\in \overline{B(0,R)}$ and $k\geq R$. 
In addition to that, this rate is optimal.
\end{cor}

When $H(x,p) = K(p)+V(x)$, the analysis becomes much more complicated due to the interaction between $K$ and $V$. We provide an example where the exponential rate of convergence is obtained in Example \ref{ex22}.

\medskip

For convex Hamiltonians, we are able to establish the exponential rate of convergence using optimal control theory.  
Some examples for which the exponential rate is obtained are given in Proposition \ref{example1} and Proposition \ref{ex33}.

\begin{thm}\label{thm:exp-rate} 
Under the assumptions $\mathrm{(P1)}, \eqref{H0},\eqref{H1}, \eqref{H3b}, \eqref{H4}$, and $\mathrm{(H5)}$, we have 
\begin{itemize}
\item[(i)] $u(x) \leq u_k(x)$ for every $x\in \overline{B(0,k)}$,
\item[(ii)] for each fixed $x\in B(0,k)$ we have
\begin{equation}\label{exp:rate}
u_k(x)  \leq u(x)+ \left(C e^{\frac{|x|}{C}}\right) e^{-\frac{k}{C}},
\end{equation}
where $C$ is a constant depending only on the growth of $H$.
\end{itemize}
In particular, for any fixed $R>0$, we have
\begin{equation*}
0 \leq u_k(x) - u(x)\leq \left(C e^{\frac{R}{C}}\right) e^{-\frac{k}{C}}
\end{equation*}
for all $x\in \overline{B(0,R)}$ and $k>R$.
\end{thm}

As a byproduct, we prove the existence of a minimizer $\eta$ with bounded velocity to the minimizing problem \eqref{rep_for} for each given $x\in \R^n$, which is a key element in 
 the proof of Theorem \ref{thm:exp-rate}.
 Moreover, the bound on the velocity of $\eta$ only depends on the growth of $H$ and not on its smoothness.
 We believe that this bound (Theorem \ref{Existence of minimizer} and Lemma \ref{bound_on_doteta})  is new in the literature.  
 See Remark \ref{rem: bounded speed} for further discussions.

\medskip

For the second prototype (P2), we establish the rate $\mathcal{O}\left(\frac{1}{k}\right)$ for a quite general class of Hamiltonians. 
The rate is also optimal, as pointed out in Remark \ref{rem: optimal for bounded domain}.

\begin{thm}\label{Conv-bdd2} 
Under assumptions $\mathrm{(P2)}, \eqref{H0}, \eqref{H1},\eqref{H3c}$ and \eqref{H4}, for any $k\geq 2$,
\begin{equation*}
0\leq u_k(x) - u(x) \leq \frac{C}{k}
\end{equation*}
for every $x\in \overline{B\left(0,1-\frac{1}{k}\right)}$ where $C$ is a constant depending only on $H$.
Moreover, this rate is optimal.
\end{thm}

Although we only deal with two prototype cases (P1) and (P2) in this paper, the obtained results can be extended to more general domains in a similar fashion under some appropriate conditions.
See Remarks \ref{rem:P2} and \ref{rem:P1} for example.

\subsection{Organization of the paper} 
The paper is organized in the following way. 
We first provide some results on state-constraint Hamilton-Jacobi equations needed throughout the paper in Section \ref{sec:prelim}. 
Section \ref{P1-doubling} and \ref{Sec:a(x)K(p)} are devoted to proving Theorem \ref{thm:approx1/k^2} and Theorem \ref{npo}, respectively. 
In the following section, we deal with the rate of convergence for convex Hamiltonians (Theorem \ref{Conv-bdd2}). 
In Section \ref{P2-bdd}, the second prototype case is considered. 
We provide some examples and further discussion in Section \ref{Sec:discussion}. 
The proofs for some results concerning minimizers of the corresponding optimal control problem are provided in Appendix.

\subsection*{Acknowledgement}
We would like to thank the two referees very much for carefully reading our manuscript and giving very helpful comments to improve the presentation of the paper.

\section{Preliminaries} \label{sec:prelim}
For an open subset $\Omega \subset\mathbb{R}^n$, we denote the space of  bounded uniformly continuous functions defined in $\Omega$ by $\mathrm{BUC}(\Omega;\mathbb{R})$. %By $C^1(\overline{\Omega})$ we mean the set of continuously differentiable functions on an open neighborhood of $\overline{\Omega}$.

\begin{defn} We say
\begin{itemize}
\item[(i)] $v\in \mathrm{BUC}(\Omega;\mathbb{R})$ is a viscosity subsolution of \eqref{HJ-static} in $\Omega$ if for every $x\in \Omega$ and $\varphi\in \mathrm{C}^1(\Omega)$ such that $v-\varphi$ has a local maximum over $\Omega$ at $x$ then $v(x) + H\big(x,D\varphi(x)\big) \leq 0$.

\item[(ii)] $v\in \mathrm{BUC}(\overline\Omega;\mathbb{R})$ is a viscosity supersolution of \eqref{HJ-static} on $\overline\Omega$ if for every $x\in \overline{\Omega}$ and $\varphi\in \mathrm{C}^1(\overline{\Omega})$ such that $v-\varphi$ has a local minimum over $\overline{\Omega}$ at $x$ then $v(x) + H\big(x,D\varphi(x)\big) \geq 0$.
\end{itemize}

If $v$ is a viscosity subsolution to \eqref{HJ-static} in $\Omega$, and is a viscosity supersolution to \eqref{HJ-static} on $\overline{\Omega}$, that is,
\begin{equation}\label{state-def}
\begin{cases}
v(x) + H(x,Dv(x)) &\leq 0 \qquad\text{in}\; \Omega,\\
v(x) + H(x,Dv(x)) &\geq 0 \qquad\text{on}\; \overline{\Omega}
\end{cases}
\end{equation}
in the viscosity sense, then we say that $v$ is a state-constraint viscosity solution of \eqref{HJ-static}. 
\end{defn}

\begin{rem} 
As pointed out in \cite{Soner1986}, the state-constraint implicitly imposes a boundary condition to solutions. 
Indeed, when $\partial\Omega$ is smooth, we can define an outward normal vector $\vec{\nu}(x)$ at $x\in \partial\Omega$. Moreover, if the state-constraint solution $v\in \mathrm{C}^1(\overline{\Omega})$, then $v$ solves $v(x)+H(x,Dv(x)) = 0$ in $\Omega$ and satisfies
\begin{equation*}
H(x,Dv(x)) \leq H\big(x,Dv(x)+\beta \vec{\nu}(x)\big) \qquad\text{for any}\;\beta \geq 0, x\in \partial\Omega.
\end{equation*}

If $H$ is differentiable in $p$, the above condition can also be phrased as a constraint on the normal derivative  on the boundary as
\begin{equation}\label{motiv}
D_pH\big(x,Dv(x)\big)\cdot\vec{\nu}(x) \geq 0 \qquad\text{for any}\; x\in \partial\Omega.
\end{equation}
\end{rem}

%%%%%ADD PROOF%%%%%%%%%%%%%%%%%%%%%%%%%%%%%5

Now we construct a state-constraint viscosity solution to \eqref{HJ-static} based on Perron's method. It is a variant of the classical result in \cite{ishii1987} but we include the proof here for the sake of the readers' convenience.
Note that the Lipschitz regularity of subsolutions is encoded directly into the admissible class $\mathcal F$.

\begin{defn} For a real valued function $w(x)$ define for $x\in \Omega$, we define the super-differential and sub-differential of $w$ at $x$ as
\begin{align*}
D^{+}w(x)&=\left \lbrace p \in \R^n : \limsup_{y \rightarrow x} \frac{w(y)-w(x)-p\cdot (y-x)}{|y-x|} \leq 0\right\rbrace,\\
D^{-}w(x)&=\left \lbrace p \in \R^n : \liminf_{y \rightarrow x} \frac{w(y)-w(x)- p \cdot (y-x)} {|y-x|} \geq 0 \right\rbrace.
\end{align*}
\end{defn}

\begin{thm}\label{Perron} Assume \eqref{H0}, \eqref{H1} and \eqref{H4}. 
Then, there exists a state-constrained viscosity solution $u\in \mathrm{C}(\overline{\Omega})\cap \mathrm{W}^{1,\infty}(\Omega)$ to \eqref{state-def}.
\end{thm}

\begin{proof}[Proof of Theorem \ref{Perron}] Under \eqref{H0} and \eqref{H1}, $C_1$ and $-C_2$ are a supersolution on $\overline{\Omega}$ and a subsolution in $\Omega$ of \eqref{HJ-static}, respectively. By the coercivity assumption \eqref{H4}, we can find a constant $C_3>0$ such that 
\begin{equation*}
H(x,p) \leq \max\{C_1,C_2\} \quad\text{for some}\; x\in \overline{\Omega} \qquad\Longrightarrow\qquad |p| \leq C_3.
\end{equation*}
Let us define
\begin{align*}
\mathcal{F}&= \big\lbrace w \in \mathrm{C}(\overline{\Omega})\cap \mathrm{W}^{1,\infty}(\Omega) :\;-C_2\leq w(x)\leq C_1,\; \Vert Dw\Vert_{L^\infty(\overline{\Omega})} \leq C_3,\\
&\qquad\qquad\qquad\qquad  \;\text{and}\;w\;\text{is a viscosity subsolution to}\; w(x)+H(x,Dw(x))\leq 0\;\text{in}\; \Omega  \big\rbrace
\end{align*}
and for each $x\in \overline{\Omega}$, we define
\begin{equation*}
u(x):=\sup \left\lbrace w(x): w\in \mathcal{F} \right\rbrace.
\end{equation*}
By the stability of viscosity subsolutions, we have that $u$ is a viscosity subsolution to \eqref{HJ-static} in $\Omega$. Thus, $u\in \mathcal{F}$ as well.

We now check that $u$ is a viscosity supersolution to \eqref{HJ-static} on $\overline{\Omega}$. Assume that  $u$ is not a supersolution on $\overline{\Omega}$. Then, there exists $x_0\in \overline{\Omega}$, $\varphi\in \mathrm{C}^1(\overline{\Omega})$ with $\|D\varphi\|_{L^\infty(B(x_0,r)} \leq C_3$ and $r>0$ such that $u(x_0) = \varphi(x_0)$ and $(u-\varphi)(x)\geq |x-x_0|^2$ for all $x\in B(x_0,r)\cap \overline{\Omega}$, and 
\begin{equation}\label{a.Perron}
\varphi(x_0) + H(x_0,D\varphi(x_0)) < 0.
\end{equation}
From \eqref{H0} and \eqref{a.Perron}, we obtain $\varphi(x_0)= u(x_0) < C_1$. By continuity of $\varphi$ and $H$, one can choose $\delta,\varepsilon\in \left(0,\frac{r}{2}\right)$ small enough so that $\varepsilon < \delta^2$ and
\begin{equation*}
\begin{cases}
\varphi(x)+\delta^2 < C_1,\\
\varphi(x) + \delta^2 + H(x,D\varphi(x)) < 0 
\end{cases} \qquad\Longrightarrow\qquad \begin{cases}
\varphi(x)+\varepsilon^2 < C_1,\\
\varphi(x) + \varepsilon^2 + H(x,D\varphi(x)) < 0 
\end{cases}
\end{equation*}
for all $x\in B(x_0,2\varepsilon)\cap \overline{\Omega}$. Clearly, $x\mapsto \varphi(x) + \varepsilon^2$ is a viscosity subsolution to \eqref{HJ-static} in $B(x_0,2\varepsilon)\cap \overline{\Omega}$ and $u(x) \geq \varphi(x) + \varepsilon^2$ for $x\in B(x,2\varepsilon)\backslash B(x_0,\varepsilon)$. Let us define $w:\overline{\Omega}\rightarrow \mathbb{R}$ by
\begin{equation*}
w(x) = \begin{cases}
\max\left\lbrace u(x), \varphi(x) + \varepsilon^2\right\rbrace &\qquad x\in B(x_0,\varepsilon)\cap \overline{\Omega},\\
u(x) &\qquad x\in \overline{\Omega}\backslash B(x_0,\varepsilon).
\end{cases}
\end{equation*}
Then, $w(x) = \max\left\lbrace u(x),\varphi(x) + \varepsilon^2\right\rbrace$ in $B(x_0,2\varepsilon)\cap \overline{\Omega}$ belongs to $\mathcal{F}$. Therefore, $w(x)$ is a viscosity subsolution to \eqref{HJ-static}. However, $w(x_0) = \varphi(x_0)+\varepsilon^2 = u(x_0)+\varepsilon^2 > u(x_0)$, which is a contradiction to the definition of $u$.
\end{proof}

The argument used in the proof of Perron's method implies the following corollary as well (see also \cite{Capuzzo-Dolcetta1990}).

\begin{cor}\label{cor:Inv_max} Let $u\in \mathrm{C}(\overline{\Omega})$ be a viscosity subsolution to \eqref{HJ-static} in $\Omega$. 
Assume further that $v\leq u$ on $\overline{\Omega}$ for all viscosity subsolutions $v\in \mathrm{C}(\overline{\Omega})$ of \eqref{HJ-static} in $\Omega$. 
Then, $u$ is a viscosity supersolution to \eqref{HJ-static} on $\overline{\Omega}$.
\end{cor}

The uniqueness of \eqref{state-def} follows from the comparison principle. It was first studied by M. Soner in \cite{Soner1986} under the following assumption on $\partial\Omega$:
 \begin{itemize}
\item[(A)] There exists a universal pair $(r,h)\in (0,\infty)\times (0,\infty)$ and a uniformly bounded continuous function $\eta\in \mathrm{BUC}\left(\overline{\Omega};\mathbb{R}^n\right)$ such that
\begin{equation}\label{A1}
B\big(x+t\eta(x),rt\big) \subset \Omega \qquad\text{for all}\; x\in \partial \Omega,\; t\in (0,h]. \tag{A}
\end{equation}
\end{itemize}
See also \cite{Capuzzo-Dolcetta1990} for other conditions to establish the comparison principle.

\begin{thm}\label{CP continuous} 
Assume $\eqref{A1}$.
 If $v_1\in \mathrm{BUC}(\overline{\Omega};\mathbb{R})$ is a viscosity subsolution of \eqref{HJ-static} in $\Omega$, and $v_2\in \mathrm{BUC}(\overline{\Omega};\mathbb{R})$ is a viscosity supersolution of \eqref{HJ-static} on $\overline{\Omega}$. If either 
\begin{itemize}
\item \eqref{H3a} holds, or
\item \eqref{H3b} holds and $v_2$ is Lipschitz,
\end{itemize}
then $v_1(x)\leq v_2(x)$ for all $x\in \overline{\Omega}$.
\end{thm}
When the uniqueness of \eqref{state-def} is guaranteed, the unique viscosity solution to \eqref{state-def} is the maximal viscosity subsolution of \eqref{HJ-static}. This property will play a crucial role in dealing with the second prototype $(\mathrm{P2})$.

%more more papers

\section{A rate of convergence for general Hamiltonians in unbounded domain}\label{P1-doubling}

In this section, we consider the first prototype (P1). 
The assumptions \eqref{H0}, \eqref{H1}, \eqref{H3c} and \eqref{H4} are enforced throughout the section. 
By Theorems \ref{Perron} and \ref{CP continuous}, there exists $u_k\in \mathrm{Lip}(\overline{B(0,k)})$ which is the unique solution to 
\begin{equation}\label{u_k}
\begin{cases}
u_k(x) + H(x,Du_k(x)) \leq 0 &\quad \text{in}\;B(0,k),\\
u_k(x) + H(x,Du_k(x)) \geq 0 &\quad \text{on}\;\overline{B(0,k)}
\end{cases}
\end{equation}
in the viscosity sense. 
Based on the construction of solutions via Perron's method together with the coercivity of $H$, we have the following a priori estimate:
\begin{equation*}
|u_k(x)|+|Du_k(x)| \leq C_H
\end{equation*}
for all $x\in B(0,k)$ in the viscosity sense. Here, $C_H$ is a positive constant depending only on $H$ (one can take $C_H = \max\{C_1,C_2,C_3\}$ from Theorem \ref{Perron}). By the Arzel\`a--Ascoli theorem, there is a subsequence $\{k_m\} \to \infty$, and a function $u\in \mathrm{Lip}(\R^n)$ such that 
\begin{equation}\label{u_def_nonconvex}
u_{k_m}\rightarrow u \qquad\text{locally uniformly in}\; \mathbb{R}^n.
\end{equation}

\begin{thm} The function $u$ defined in \eqref{u_def_nonconvex} is a viscosity solution to
\begin{equation}\label{u_eqn}
u(x) + H(x,Du(x)) = 0 \qquad\text{in}\;\mathbb{R}^n.
\end{equation}
Moreover, $u_k\rightarrow u$ locally uniformly in $\mathbb{R}^n$ as $k$ grows to infinity.
\end{thm}

\begin{proof} 
It is clear from the stability of viscosity solutions that $u$ is a solution to \eqref{u_eqn}. 
The fact that $u_k\rightarrow u$ locally uniformly in $\R^n$ follows from the uniqueness of  solutions to \eqref{u_eqn}.
\end{proof}

Now we are ready to give a proof for Theorem \ref{thm:approx1/k^2} using the doubling variables method.

\begin{proof}[Proof of Theorem \ref{thm:approx1/k^2}] 
We first note that $u_k$ solves $u_k(x)+H(x,Du_k(x)) \geq 0$ on $\overline{B(0,k)}$, and $u$ solves $u(x)+H(x,Du(x)) \leq 0$ in $B(0,k)$ in viscosity sense. By the comparison principle, we get $u_k(x) \geq u(x)$ for all $x\in \overline{B(0,k)}$.

For the upper bound of $u_k - u$, we define the following auxiliary function
\begin{equation*}
\Phi^k(x,y) = u_k(x) - u(y) - 2C_H k^{2}\left|x-y\right|^2 - \frac{8C_H}{k^{2}}|y|^2
\end{equation*}
for $(x,y)\in \overline{B(0,k)}\times\mathbb{R}^n$. 
It is clear that $\Phi^k$ is bounded above by $2C_H$ independent of $k\in \mathbb{N}$. 
If $|y|>\frac{k}{2}$, then we have
\begin{align*}
\Phi^k(0,0)-\Phi^k(x,y)  \geq -u_k(x) +{u_k(0) - u(0)} + u(y)+2C_H k^{2}|x-y|^2 + \frac{8C_H}{k^{2}}|y|^2>0,
\end{align*}
which implies that for each $k\in \mathbb{N}$,  $\Phi^k(x,y)$ achieves a global maximum over $\overline{B(0,k)}\times \mathbb{R}^n$ at $(x_k,y_k)\in \overline{B(0,k)}\times \overline{B\left(0,\frac{k}{2}\right)}$. 
Of course, $|y_k| \leq \frac{k}{2}$. 
Now we use $\Phi^k(x_k,y_k)\geq \Phi^k(y_k,y_k)$ to get
\begin{equation*}
2C_Hk^{2}|x_k-y_k|^2 \leq u_k(x_k) - u_k(y_k) \leq C_H|x_k-y_k|.
\end{equation*}
Therefore,  we deduce that
\begin{equation}\label{nba1}
|x_k| \leq  |y_k| + \frac{1}{2k^{2}} < k
\end{equation}
for all $k\geq 1$ since $|y_k| \leq \frac{k}{2}$. 
Observing that $x\mapsto \Phi^k(x,y_k)$ obtains a maximum at $x_k$ with $|x_k|<k$,  we have
\begin{equation}\label{naa.sub}
u_k(x_k) + H\left(x_k,p_k\right) \leq 0,
\end{equation}
where $p_k = 4C_H k^{2}(x_k-y_k)$ by the definition of viscosity subsolutions. 
We also observe that $y\mapsto \Phi^k(x_k,y)$ obtains a maximum at $y_k$, which implies that 
\begin{equation*}
u(y) - \left(-2C_Hk^{2}\left|x_k-y_k\right|^2 - \frac{8C_H}{k^{2}}|y|^2\right)
\end{equation*}
has a minimum at $y_k$. 
By the definition of viscosity supersolutions, we get
\begin{equation}\label{naa.super}
u(y_k) + H(y_k,p_k+q_k) \geq 0
\end{equation}
where $q_k = -\frac{16C_H}{k^{2}} y_k$. 
Here, it needs to be noted that
\begin{equation*}
|p_k|,|p_k+q_k| \leq C_H,
\end{equation*}
which comes from Lipschitz continuity of $u_k$. 
Using \eqref{naa.sub}, \eqref{naa.super} and assumption \eqref{H3c}, there exists a constant $\tilde{C}_H$ such that
\begin{align}\label{naa.sum}
u_k(x_k) - u(y_k) &\leq H(y_k,p_k+q_k) - H(x_k,p_k)\nonumber\\
&= H(y_k,p_k+q_k) - H(y_k,p_k) + H(y_k,p_k) - H(x_k,p_k) \nonumber\\
&\leq \tilde{C}_H|q_k| + \tilde{C}_H|x_k-y_k| \nonumber\\
&\leq \frac{16\tilde{C}_HC_H}{k^{2}}|y_k| + \frac{\tilde{C}_H}{k^{2}}
\leq \frac{8 \tilde{C}_HC_H}{k} + \frac{\tilde{C}_H}{k^{2}}.
\end{align}
If we stop here, the fact that $\Phi^k(x_k,y_k) \geq \Phi^k(x,x)$ for $x \in B(0,k)$ gives 
\begin{align*}
u_k(x) - u(x)\leq u_k(x_k) - u(y_k) + \frac{8C_H}{k^{2}}|x|^2 \leq \frac{C}{k} + \frac{C(1+|x|^2)}{k^2}
\end{align*}
for all $k \geq 2$.
This gives us the rate of convergence of $u_k$ to $u$ is $\mathcal{O}(\tfrac{1}{k})$ for $x\in B(0,R)$, which is typically the case in light of the doubling variables method.

\medskip

Nevertheless, a key new point here is to bootstrap once more to improve this rate.
The monotonicity of $\{u_k\}$ allows us to bound $|y_k|$ better.
We use that $\Phi^k(x_k,y_k)\geq \Phi^k(0,0)$ together with \eqref{naa.sum} and $u_k \geq u$ to yield
\begin{align*}
2C_H k^{2}|x_k-y_k|^2 + \frac{8C_H}{k^{2}}|y_k|^2 &\leq u_k(x_k) - u_k(0) + u(0)- u(y_k) \\
&\leq u_k(x_k) - u(y_k)\\
&\leq \frac{16\tilde{C}_HC_H}{k^2}|y_k| + \frac{\tilde{C}_H}{k^{2}}.
\end{align*}
Therefore, 
\begin{equation*}
|y_k|^2 \leq 2\tilde{C}_H |y_k|+\frac{\tilde{C}_H}{8C_H} \leq \frac{1}{2} |y_k|^2 + 2\tilde C_H^2 + \frac{\tilde{C}_H}{8C_H} = \frac{1}{2} |y_k|^2 + C.
\end{equation*}
In particular, $|y_k| \leq C$.
This bound is much better than the earlier bound that $|y_k| \leq \frac{k}{2}$.

\medskip

Now for any $x\in \overline{B(0,k)}$, clearly we have that $\Phi^k(x_k,y_k) \geq \Phi^k(x,x)$.
This, together with \eqref{naa.sum} and $|y_k| \leq C$, implies
\begin{align*}
u_k(x) - u(x)\leq u_k(x_k) - u(y_k) + \frac{8C_H}{k^{2}}|x|^2 \leq \frac{C(1+|x|^2)}{k^2}
\end{align*}
for all $k \geq 2$. 
If $|x|\leq R$, then
\begin{equation*}
0\leq u_k(x) - u(x) \leq \frac{C(1+R^2)}{k^2},
\end{equation*}
which gives the desired result.
\end{proof}

\begin{rem} 
In the general setting, one only has that $\Phi^k(x,y)$ achieves a global maximum over $\overline{B(0,k)}\times \mathbb{R}^n$ at $(x_k,y_k)$ where $|y_k| \leq \frac{k}{2}$ and $|x_k| < k$.
In our current situation, the monotonicity of $\{u_k\}$ allows us to bootstrap once more to deduce further that $|y_k| \leq C$,
which helps to obtain $\mathcal{O}\left(\frac{1}{k^2}\right)$ rate of convergence.
This seems to be the best convergence rate that one is able to get through the doubling variables method here as it is unlikely that $|y_k|$ vanishes as $k \to \infty$.

We do not know yet whether the $\mathcal{O}\left(\frac{1}{k^2}\right)$ rate of convergence is optimal or not in the general nonconvex setting.
See Questions \ref{Q1} and \ref{Q2} in Section \ref{Sec:discussion} below.
%The key estimate in the bootstrap argument relies on the fact that $-u_k(0)+u(0)\geq 0$. 
%It definitely helps us to get $\mathcal{O}\left(\frac{1}{k^2}\right)$ rate of convergence in this situation. 
%One could get better rate if $|y_k|$ vanished as $k$ goes to infinity but it is not quite plausible using the method presented above. 
\end{rem}

\section{An optimal rate for a class of nonconvex Hamiltonians on unbounded domain}\label{Sec:a(x)K(p)}

In this section, we show that the rate of convergence $u_k\rightarrow u$ is of order $\mathcal{O}(e^{-Ck})$ for a class of possibly nonconvex Hamiltonians which are written as $H(x,p) = a(x)K(p)$ with $K(0)=0$ and $0<\alpha \leq a(x) \leq \beta$. The aforementioned rate is indeed optimal. 

A brief idea for the proof is that we construct a supersolution to \eqref{u_k} by finding a symmetric Hamiltonian $\tilde{H}$ such that $\tilde{H}(0) = 0$ and $\tilde{H}\leq H$. The following proposition is needed as a building block.

\begin{prop}\label{prop:linear} Let $H:\mathbb{R}^n\rightarrow\mathbb{R}$ be defined by
\begin{equation*}
H(p) = \begin{cases}
-\alpha|p| & \text{for}\; |p|\leq \beta,\\
\;\;\; f(p)  &\text{for}\; |p|\geq \beta,
\end{cases}
\end{equation*}
where $\alpha,\beta>0$ and $f:\mathbb{R}^n\rightarrow\mathbb{R}$ is a coercive continuous function such that $f(p) = -\alpha \beta$ for $|p|= \beta$ and $\min_{\mathbb{R}^n} f = -\alpha \beta$. Then,
\begin{equation*}
u_k(x) = \alpha \beta e^{\frac{|x|-k}{\alpha}}
\end{equation*}
for $x\in \overline{B(0,k)}$ is the unique solution to the state-constraint problem \eqref{u_k}.
\end{prop}

\begin{proof} It is clear that $u_k(x) + H(Du_k(x)) = 0$ in $B(0,k)\backslash \{0\}$ in classical sense. For $x \in \partial B(0,k)$ and $\varphi\in \mathrm{C}^1(\overline{B(0,k)})$ such that $u_k-\varphi$ has a local minimum over $\overline{B(0,k)}$ at $x$, we have $u_k(x) + H(D\varphi(x)) \geq 0$ since $u_k(x) = \alpha \beta = -\min H$. We only need to check if $u_k$ is a viscosity supersolution at $x = 0$. 

Let $\varphi \in \mathrm{C}^1(\mathbb{R}^n)$ such that $\varphi(0) = u_k(0)$ and $u_k-\varphi$ has a local minimum over $\overline{B(0,k)}$ at $ x = 0$. Since $u_k$ is convex, we can replace $\varphi$ by an affine function $\varphi(x) =\xi \cdot x+ u_k(0)$ for some $\xi  \in \mathbb{R}^n$. Without loss of generality, it suffices to consider $\xi\neq 0$. For $|x|$ sufficiently small, we have $u_k(x) -\varphi(x) \geq u_k(0) - \varphi(0)$, which implies that
\begin{equation}\label{eqn:ex1.1}
\alpha\beta e^{-\frac{k}{\alpha}}\big(e^{\frac{|x|}{\alpha}}-1\big) \geq  \xi\cdot x.
\end{equation}

Now we choose $x = t\frac{\xi}{|\xi|}$ for $t>0$ small, then \eqref{eqn:ex1.1} implies that $\alpha\beta e^{-\frac{k}{\alpha}}\big(e^{\frac{t}{\alpha}}-1\big) \geq t|\xi|$ for all $t>0$ sufficiently small. Dividing both sides by $t$ and sending $t$ to 0, we deduce that $|D\varphi(0)| = |\xi|  \leq \beta e^{-\frac{k}{\alpha}}$. Therefore,
\begin{equation*}
u_k(0) + H(D\varphi(0)) = \alpha \beta e^{-\frac{k}{\alpha}} -\alpha |D\varphi(0)| \geq 0.
\end{equation*}
Consequently, $u_k$ is the unique viscosity  solution to \eqref{u_k}.
\end{proof}

\begin{proof}[Proof of Theorem \ref{npo}] Since $K(0) = 0$, $u\equiv 0$ is the unique solution to \eqref{u_eqn}. Recalling the a priori estimate $\Vert u_k\Vert_{L^\infty(B(0,k))}+\Vert Du_k\Vert_{L^\infty(B(0,k))}\leq C_H$, condition \eqref{H3c} gives
\begin{equation*}
|K(p) - K(q)| \leq L|p-q|
\end{equation*}
for all $p,q\in \overline{B(0,C_H)}$. Let $K(p_0) = \min K \leq 0$ for some $p_0\in \mathbb{R}^n$. 
Let $f(p)$ be a coercive, continuous function such that $f(p)=-L|p_0|$ for $|p| \leq |p_0|$, $\min_{\R^n} f = - L|p_0|$, and $f(p) \leq K(p)$ for $|p| \geq |p_0|$. Now we consider
\begin{equation*}
\tilde{H}(p) = \begin{cases}
-  L|p|                  &\quad\text{for}\; |p|\leq |p_0|,\\
\;\;\; f(p)  &\quad\text{for}\; |p|\geq |p_0|.
\end{cases}
\end{equation*}

\begin{figure}[h]
\begin{center}
\includegraphics[scale=0.4]{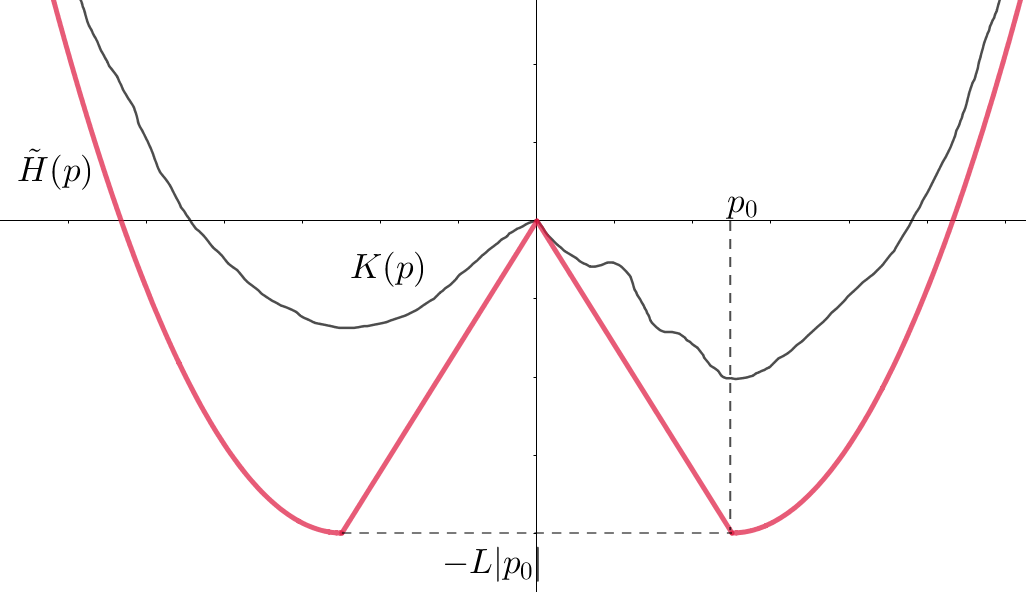}
\end{center}
\caption{The graph of $\tilde{H}(p)$ and $K(p)$.}
\label{pic:1}
\end{figure}
%\FloatBarrier

The graph of $\tilde{H}$ is described in Figure \ref{pic:1}. It is clear that $\tilde{H}(p)\leq K(p)$ for all $p\in \mathbb{R}^n$. Moreover, using Proposition \ref{prop:linear}, the unique viscosity solution to the state-constraint problem $\tilde{u}_k(x) + \beta\tilde{H}(D\tilde{u}_k(x)) = 0$ in $B(0,k)$ is given by $\tilde{u}_k(x) = \beta L|p_0|e^{\frac{|x|-k}{\beta L}}$ for $x\in \overline{B(0,k)}$. 
\medskip

It is clear that $\tilde{u}_k$ is also the unique viscosity solution to $\frac{1}{\beta} \tilde{u}_k(x) + \tilde{H}(D\tilde{u}_k(x)) = 0$ in $B(0,k)$. Since $\beta \geq a(x)\geq \alpha >0$ and $\tilde{H}\leq K$, we deduce that
\begin{equation*}
\tfrac{1}{\beta} \tilde{u}_k(x) + \tilde{H}(D\tilde{u}_k(x)) \leq \tfrac{1}{a(x)} \tilde{u}_k(x) + K(D\tilde{u}_k(x))
\end{equation*}
on $\overline{B(0,k)}$. Therefore, $\tilde{u}_k(x)+ a(x)K(D\tilde{u}_k(x)) \geq 0$ on $\overline{B(0,k)}$. By the comparison principle, one gets
\begin{equation*}
0\leq u_k(x) \leq \beta L |p_0| e^{\frac{|x|-k}{\beta L}}
\end{equation*}
for all $x\in \overline{B(0,k)}$. The conclusion for $|x|\leq R$ follows immediately.
\end{proof}

In case that $H(x,p)=K(p)$ for $(x,p) \in \R^n \times \R^n$, where $K$ is locally Lipschitz continuous and coercive in $\R^n$, we have the unique  viscosity solution  to \eqref{u_eqn} is $u\equiv -K(0)$. 
Therefore, we can assume that $K(0) = 0$, and Corollary \ref{cor1} follows without assuming that $K(0) = 0$. 

\medskip

It should be noted that the local Lipschitz continuity of Hamiltonians is important when it comes to getting an exponential rate of convergence. If a Hamiltonian is only H\"older continuous around $0$, we get a slower rate of convergence depending on the regularity of $H$ as described in the following proposition. 

\begin{prop} \label{example3} Let $H:\mathbb{R}^n\rightarrow\mathbb{R}$ defined by
\begin{equation*}
H(p) = \begin{cases}
-|p|^\gamma \qquad & \text{if}\;\;|p|\leq 1,\\
\;f(p)  \qquad &\text{if}\;\; |p|\geq 1,
\end{cases}
\end{equation*}
where $\gamma\in(0,1)$ and $f:\mathbb{R}^n\rightarrow\mathbb{R}$ is a continuous, coercive function with $f(x) = -1$ for $|x|\leq 1$, and $\min_{\mathbb{R}^n} f = -1$. Then, the solution to \eqref{u_k} is given by
\begin{equation}\label{eqn:ex3}
u_k(x) = \left[\tfrac{1-\gamma}{\gamma}\left(k+\tfrac{\gamma}{1-\gamma} - |x|\right)\right]^\frac{\gamma}{\gamma-1}, \qquad x\in \overline{B(0,k)}.
\end{equation}
As a consequence, $u_k\rightarrow 0$ with the rate $\mathcal{O}\left(\frac{1}{k^{\frac{\gamma}{1-\gamma}}}\right)$.
\end{prop}

\begin{proof} 
Let us first consider the one dimensional case. 
The higher dimensional setting can be done in a same manner. 
Let $\mu = \gamma^{-1}
$, we look for a nonnegative solution to $u(x)^\mu = u'(x)$ where $x\in (0,k)$. We have
\begin{equation*}
u(x)^{1-\mu} = (1- \mu)x-C_k \qquad\Longrightarrow\qquad u(x) = (\mu-1)^\frac{1}{1-\mu}(C_k-x)^\frac{1}{1-\mu}.
\end{equation*}
We want to choose $C_k$ such that $u'(x)\in [0,1]$ for $x\in (0,k)$. Equivalently,
\begin{equation*}
u'(x) = u(x)^{\mu}= (\mu-1)^\frac{\mu}{1-\mu}(C_k-x)^\frac{\mu}{1-\mu} \in [0,1]
\end{equation*}
for $x\in (0,k)$. Since it is an increasing function, $C_k = k + \frac{1}{\mu-1}$. Using symmetry, we guess that $u_k$ is written as
\begin{equation*}
u_k(x) = (\mu-1)^\frac{1}{1-\mu} \left(k+\tfrac{1}{\mu-1}-|x|\right)^\frac{1}{1-\mu}.
\end{equation*}
It is straightforward to see that $u_k$ satisfies the equation in the classical sense $u_k(x) -|u_k'(x)|^\gamma = 0$ in $(-k,k)\backslash \{0\}$. Since $|u_k'(x)|\leq 1$ on $(-k,k)\backslash \{0\}$, we have $u_k(x) + H(u_k'(x)) = 0$ in the classical sense in $(-k,k)\backslash \{0\}$. At $|x| = k$, we have $u_k(x)=1 \geq -\min H$. Therefore, the supersolution test at these points are satisfied. Finally, at $x=0$ we only need to verify the supersolution test, which is simple since if $p\in D^-u_k(0)$ then
\begin{equation*}
|p|\leq (\mu-1)^\frac{\mu}{1-\mu} \left(k+\tfrac{\mu}{\mu-1}\right)^\frac{\mu}{1-\mu} \qquad\Longrightarrow\qquad u_k(0) + H(p) \geq u_k(0) - |p|^\mu \geq 0.
\end{equation*}
Thus, $u_k$ defined above is the unique viscosity solution to the constraint problem \eqref{u_k}. Using a similar argument as in the proof of Proposition \ref{prop:linear}, this formula of $u_k$ can be extended naturally to the $n$-dimensional case, as given in \eqref{eqn:ex3} and the conclusion follows.
\end{proof}

\begin{rem} From Proposition \ref{example3} we see that the optimal rate of convergence can be as slow as we wish as the H\"older exponent $\gamma\rightarrow 0^+$.
This shows that the required condition  \eqref{H3c} is really essential in this section.
\end{rem}

When Hamiltonians are of the form $H(x,p) = K(p)+V(x)$, the situation becomes much more complicated. See Example \ref{ex22} for a situation where we get the optimal exponential rate of convergence with nonconvex $K$.

\section{An optimal rate for convex Hamiltonians}\label{5}
In this section, the assumptions $\eqref{H0}$, $\eqref{H1}$, $\eqref{H3b}$, $(\mathrm{H5})$ are always in force. The state-constraint problem was studied in the context of optimal control for convex Hamiltonians (see \cite{Soner1986, Capuzzo-Dolcetta1990,Bardi1997} for instance). When $H$ is convex, we are able to obtain a representation formula for the viscosity solution based on the optimal control theory. Let us assume the following superlinear property (see Remark \ref{rem:removeH6} where we can remove this assumption), which is
\begin{itemize}
\item[(H6)] $p\mapsto H(x,p)$ is superlinear uniformly for $x\in\overline{\Omega}$, that is,
\begin{equation}\label{H6}
\lim_{|p|\rightarrow \infty} \left(\inf_{x\in \Omega} \frac{H(x,p)}{|p|}\right) = +\infty. \tag{H6}
\end{equation}
\end{itemize}
If $\mathrm{(H5)}$ and \eqref{H6} hold, then the Legendre transform $L:\overline{\Omega}\times \mathbb{R}^n$ of $H$ is defined as
\begin{equation*}
L(x,v) := \sup_{p\in \mathbb{R}^n} \big\lbrace p\cdot v - H(x,p)\big\rbrace, \qquad(x,v)\in \overline{\Omega}\times \mathbb{R}^n.
\end{equation*}

\begin{lem} Assume $\mathrm{(H5)}$ and \eqref{H6}. Then, $L:\overline{\Omega}\times\mathbb{R}^n \rightarrow\mathbb{R}$ is continuous satisfying:
\begin{itemize}
\item[$\mathrm{(L1)}$] If $\mathrm{(H1})$ holds, then $L(x,0) \leq C_1$ for all $x\in \overline{\Omega}$;
\item[$\mathrm{(L2)}$] If $\mathrm{(H2}$ holds, then $L(x,v) \geq -C_2$ for all $(x,v)\in \overline{\Omega}\times\mathbb{R}^n$; 
\item[(L3)] If $\mathrm{(H3b)}$ holds, then for each $R>0$ there exists a modulus $\tilde{\omega}_R(\cdot)$ such that 
\begin{equation*}
|L(x,v) - L(y,v)| \leq \tilde{\omega}_R(|x-y|) \qquad\text{for all}\; x,y\in \overline{\Omega}, |v|\leq R.
\end{equation*}
\item[(L5)] $v\mapsto L(x,v)$ is convex for each $x\in \overline{\Omega}$;
\item[(L6)] $v\mapsto L(x,v)$ is superlinear uniformly in $x\in\overline{\Omega}$, i.e.,
\begin{equation}\label{L6}
\lim_{|p|\rightarrow \infty} \left(\inf_{x\in \overline{\Omega}} \frac{L(x,v)}{|v|}\right) = +\infty. \tag{L6}
\end{equation}
\end{itemize}
\end{lem}
We omit the proof of this lemma and refer the interested readers to \cite{PiermarcoCannarsa2004}. 

For each $x\in \overline{\Omega}$, we define the admissible set of paths as
\begin{equation*}
\mathcal{A}_x = \Big\lbrace \eta\in \mathrm{AC}\big([0,\infty);\mathbb{R}^n\big): \eta(0)=x\;\text{and}\; \eta(s)\in \overline{\Omega}\;\text{for all}\;s\geq 0 \Big\rbrace
\end{equation*}
where $\mathrm{AC}\big([0,\infty);\mathbb{R}^n\big)$ denotes the set of absolutely continuous curves from $[0,\infty)$ to $\mathbb{R}^n$. Note that $\mathcal{A}_x\neq \emptyset$ since $\eta(s)\equiv x$ for all $s\in [0,\infty)$ is an admissible path. 
From this, define the value function as
\begin{equation}\label{def.u}
u(x):= \inf_{\eta\in \mathcal{A}_x} J\left[x,\eta\right] 
\end{equation}
where the cost functional is defined as
\begin{equation*}
J\left[x,\eta\right] = \int_0^\infty e^{-s}L\big(\eta(s),-\dot{\eta}(s)\big)\;ds
\end{equation*}
for $\left(x,\eta\right)\in \overline{\Omega} \times \mathcal{A}_x$. Now we have the following classical dynamic programming principle.

\begin{thm}[Dynamic Programming Principle]\label{DPP} For any $t>0$, we have 
\begin{equation*}
u(x) = \inf_{\eta\in \mathcal{A}_x} \left\lbrace \int_0^t e^{-s}L\big(\eta(s),-\dot{\eta}(s)\big)\;ds + e^{-t}u\big(\eta(t)\big) \right\rbrace.
\end{equation*}
\end{thm}

Using the Dynamic Programming Principle, one can prove that $u\in \mathrm{BUC}(\overline {\Omega})$ and indeed a viscosity solution to \eqref{state-def} as stated in the following theorems. 

\begin{thm}\label{thm:contu} Assume \eqref{H0}, \eqref{H1}, \eqref{H3c}, $(\mathrm{H5})$ and \eqref{H6}, then the function $u(x)$ defined by \eqref{def.u} is bounded and is uniformly continuous up to the boundary, which is $u\in \mathrm{BUC}(\overline{\Omega})$.
\end{thm}

\begin{thm}\label{thm:HJ} The value function $u\in \mathrm{BUC}(\overline{\Omega})$ defined in \eqref{def.u} is a viscosity solution to the state-constraint Hamilton-Jacobi equation $u(x) + H(x,Du(x)) = 0$ in $\Omega$, i.e.,
\begin{equation*}
\begin{cases}
u(x) + H(x,Du(x)) \leq 0 \qquad\text{in}\;\Omega,\\
u(x) + H(x,Du(x)) \geq 0 \qquad\text{on}\;\overline{\Omega}.
\end{cases}
\end{equation*}
\end{thm}

We omit the proofs of Theorems \ref{DPP}, \ref{thm:contu} and \ref{thm:HJ}. 
We refer to  \cite{Soner1986, Capuzzo-Dolcetta1990,Bardi1997} for those who are interested. 
\medskip

On the other hand, when $\Omega = \mathbb{R}^n$, it is known that the function $u(x)$ defined in \eqref{def.u} satisfies the Hamilton-Jacobi equation \eqref{u_eqn} in viscosity sense (see \cite{Bardi1997, Le2017} for instance). 

\begin{thm} For each $x\in  \mathbb{R}^n$, we define
\begin{equation}\label{rep_for}
u(x) = \inf_{\eta\in \mathcal{A}_{x}} \int_0^\infty e^{-s}L\left(\eta(s),-\dot{\eta}(s)\right)\;ds
\end{equation}
subject to $\mathcal{A}_{x} = \left\lbrace \eta\in\mathrm{AC}\big([0,\infty);\mathbb{R}^n\big): \eta(0) = x  \right\rbrace$. Then, $u\in \mathrm{BUC}(\mathbb{R}^n)$ is a viscosity solution to \eqref{u_eqn} and we have the following priori estimate:
\begin{equation}\label{est.u}
\Vert u\Vert_{L^\infty(\mathbb{R}^n)} + \vert Du \Vert_{L^\infty(\mathbb{R}^n)} \leq C_H.
\end{equation}
\end{thm}

\begin{rem}\label{rem:removeH6} 
We may assume that $H$ is just coercive rather than superlinear. 
When a Hamiltonian is coercive, we still have that \eqref{est.u} holds for some $C=C_H>0$. 
Therefore, for $|p| \geq C$, we can modify $H$ so that \eqref{H6} holds. Furthermore, we can impose a quadratic growth rate on $H$ as following.
\begin{itemize}
\item[(H7)] There exist some positive constants $A,B$ such that
\begin{equation}\label{H7}
A^{-1}|v|^2 - B \leq H(x,p) \leq A|v|^2 + B \qquad\text{for}\; (x,p)\in \mathbb{R}^n\times\mathbb{R}^n. \tag{H7}
\end{equation}
\end{itemize}
It is easy to see from $\mathrm{(H7)}$ that we have $(4A)^{-1}|v|^2 - B \leq L(x,v) \leq 4A|v|^2+B$ for all $(x,v)\in \mathbb{R}^n\times\mathbb{R}^n$. By making $A$ bigger, we can assume the following.
\begin{itemize}
\item[(L7)] 
\begin{equation}\label{L7}
A^{-1}|v|^2 - B \leq L(x,v) \leq A|v|^2 + B \qquad\text{for}\;(x,v)\in \mathbb{R}^n\times\mathbb{R}^n. \tag{L7}
\end{equation}
\end{itemize}
\end{rem}

We give a proof for the existence of a minimizer with bounded velocity to \eqref{rep_for} for the sake of readers' convenience in Appendix. 
This is an extremely important fact in our analysis and is a key element in the proof of Theorem \ref{thm:exp-rate} (see Remark \ref{rem: bounded speed} for further discussions).
To establish this point, the following lemma on the subdifferentials of $L(x,v)$ in $v$ is needed. For continuously differentiable Lagrangians, it is obvious, but we state here a slightly more general version.

\begin{lem}\label{growth-L'-lem} Let $L:\mathbb{R}^n\times\mathbb{R}^n \rightarrow\mathbb{R}$ be continuous and satisfy $\mathrm{(L5)}$ and \eqref{L7}. There exists $C_L >0$ such that for all $v\in \mathbb{R}^n$, we have
\begin{equation}\label{growth-L'}
 |\xi|\leq C_L(1+|v|) \qquad\text{whenever}\;\xi\in D^-_v L(x,v).
\end{equation}
\end{lem}

For simplicity, let us assume further that 
\begin{itemize}
\item[(L8)] $(x,v)\mapsto L(x,v)$ is continuously differentiable on $\mathbb{R}^n\times \mathbb{R}^n$.
\end{itemize}
This assumption can be removed in the proof of Theorem \ref{thm:exp-rate} due to the fact that the estimate \eqref{exp:rate} does not depend on the regularity of $H$, hence, we can approximate $H$ by convex, smooth Hamiltonians.

\begin{thm}[Existence of a minimizer]\label{Existence of minimizer} Let $L(x,v)$ be a continuous Lagrangian satisfying $\mathrm{(L5)}$, \eqref{L7} and $\mathrm{(L8)}$. Then, for each $x\in \mathbb{R}^n$, there exists $\eta\in \mathcal{A}_{x}$ such that $J[x,\eta] = u(x)$ and also $\Vert e^{-s/2}\dot{\eta}(s)\Vert_{L^2(0,\infty)} \leq C_4$ where $C_4$ depends only on $C_H,A,B$.
\end{thm}

The existence of minimizers of smooth Lagrangian is sufficient for our proof of Theorem \ref{thm:exp-rate} since the last estimate does not depend on the smoothness of $L$ or $H$. Clearly, a minimizer for a general continuous Lagrangian can be obtained via approximation of smooth Lagrangians (see Appendix).
\medskip

A minimizer to \eqref{rep_for} satisfies the following properties.
%We also have that a minimizer $\eta$ satisfies dynamic programming principle as stated.

\begin{lem}\label{DPP_lemma} Let $x\in \mathbb{R}^n$ and $\eta$ be a corresponding minimizer. For any $t>0$, we have
\begin{equation}\label{DPP_lem}
u(x) = \int_0^t  e^{-s} L\big(\eta(s),-\dot{\eta}(s)\big)\;ds  + e^{-t}u\left(\eta(t)\right).
\end{equation}
Furthermore, for every $t,h>0$, we have 
\begin{equation}\label{char_4}
u(\eta(t)) = e^{t}\int_t^\infty e^{-s}L\big(\eta(s),-\dot{\eta}(s)\big)\;ds
\end{equation}
and
\begin{equation}\label{char_5}
e^{-t}u\left(\eta(t)\right) = \int_t^{t+h}  e^{-s} L\big(\eta(s),-\dot{\eta}(s)\big)\;ds  + e^{-(t+h)}u\left(\eta(t+h)\right).
\end{equation}
\end{lem}

\begin{lem}\label{bound_on_doteta} Let $x\in \mathbb{R}^n$ and $\eta$ be a minimizer to \eqref{rep_for} associated with it. Then, there exists a constant $C_5>0$ depending only on $C_H,A,B$ such that $|\dot{\eta}(s)|\leq C_5$ for a.e. $s\in (0,\infty)$.
\end{lem}

\begin{rem}\label{remark:Euler-Lagrange} We provide here a connection between a minimizer $\eta$ of $u(x) = J[x,\eta]$ and some properties in the view of the method of characteristics. If $H$ is assumed to be $\mathrm{C}^2$, then $L\in \mathrm{C}^2$ and $\eta$ is a weak solution to the Euler-Lagrange equation 
\begin{equation}\label{E-L}
D_xL(\eta(s),-\dot{\eta}(s)) - D_vL(\eta(s),-\dot{\eta}(s)) + \frac{d}{ds}\Big(D_vL(\eta(s),-\dot{\eta}(s))\Big) = 0.
\end{equation}
Assume that $\eta\in \mathrm{C}^2$ (it holds if, for instance $L\in \mathrm{C}^{2,\alpha}$ for some $\alpha\in (0,1)$). Then, one can define the momentum $\textbf{p}(s) = D_vL(\eta(s),-\dot{\eta}(s))$ and show that
\begin{equation}\label{characteristic}
u(\eta(t)) + H\big(\eta(t),\textbf{p}(t)\big) = 0
\end{equation}
for $t>0$. Indeed, for every fixed $x\in \mathbb{R}^n$, we recall that
\begin{equation}\label{ll}
v \in D_p^-H(x,p) \quad \Leftrightarrow\quad p \in  D^-_vL(x,v) \quad\Leftrightarrow\quad H(x,p) + L(x,v) = p\cdot v.
\end{equation}
Using \eqref{ll} we can deduce that 
\begin{equation*}
\begin{cases}
\qquad\qquad\frac{d}{ds}\left(e^{-s}\textbf{p}(s)\right) &= e^{-s}D_xL\left(\eta(s),\dot{\eta}(s)\right),\\
\;\,\quad\frac{d}{ds}\left(H\left(\eta(s),\textbf{p}(s)\right)\right) &= -\dot{\eta}(s)\cdot \textbf{p}(s),\\
\frac{d}{ds}\left(e^{-s}H\left(\eta(s),\textbf{p}(s)\right)\right) &= e^{-s}L\left(\eta(s),-\dot{\eta}(s)\right).
\end{cases}
\end{equation*}
From that, we can derive the characteristic ODEs for $s>0$, which are
\begin{equation*}
\begin{cases}
-\dot{\eta}(s) &= D_pH(\eta(s),\textbf{p}(s)),\\
\;\; \, \dot{\textbf{p}}(s) &= \textbf{p}(s) - D_xL\big(\eta(s),-\dot{\eta}(s)\big).
\end{cases}
\end{equation*}
This together with \eqref{char_4} yields that
\begin{equation*}
u(\eta(t)) + H\big(\eta(t),\textbf{p}(t)\big) = Ce^t \qquad\text{where}\qquad C = \lim_{a\rightarrow \infty} e^{-a}H\big(\eta(a),\textbf{p}(a)\big).
\end{equation*}
Lemma \ref{growth-L'-lem} together with Lemma \ref{bound_on_doteta} gives us a uniform bound on $\textbf{p}$, thus $C = 0$. Hence, \eqref{characteristic} follows.
\end{rem}

Now we give a proof for Theorem \ref{thm:exp-rate}. Recall that we have the value function
\begin{equation}\label{u_kk}
u_k(x) =  \inf_{\eta\in \mathcal{A}_{x}^k} \int_0^\infty e^{-s}L\left(\eta(s),-\dot{\eta}(s)\right)\;ds,
\end{equation}
where $\mathcal{A}^k_{x} = \Big\lbrace \eta\in \mathrm{AC}\big([0,\infty);\mathbb{R}^n\big): \eta(0) = x\;\text{and}\; \eta(s)\in \overline{B(0,k)}\;\text{for}\;s\geq 0 \Big\rbrace$. Then, $u_k$ solves the state-constraint problem \eqref{u_k}.

\begin{proof}[Proof of Theorem \ref{thm:exp-rate}] Let $k \in \N$ be given. We may assume that $H$ satisfies \eqref{H6} and \eqref{H7} up to modification for $|p|$ large enough. Also, since the final estimate does not depend on the smoothness of $L$, we can assume $H$ is smooth and thus $L$ is smooth without any loss of generality. Clearly, $\mathcal{A}_{x}^k \subset \mathcal{A}_{x}$ for any $x\in \overline{B(0,k)} $, which implies that $u_k(x) \geq u(x)$. 
\medskip

For $x\in {B(0,k)}$, let $\eta\in \mathcal{A}_{x}$ be a minimizer to \eqref{rep_for}, if $\eta(s)\in \overline{B(0,k)}$ for all $s>0$, then $\eta\in \mathcal{A}_{x}^k$ as well, hence $u(x) = u_k(x)$. Otherwise, there exists $t>0$ such that $\eta(t)\in \partial B(0,k)$ and $\eta(s)\in B(0,k)$ for all $s\in (0,t)$. By Lemma \ref{bound_on_doteta}, we have
\begin{equation*}
k =|\eta(t)| \leq |\eta(0)| + \int_0^{t} |\dot{\eta}(s)|\;ds \leq |x| + C_5t,
\end{equation*}
which implies that $t\geq \frac{k-|x|}{C_5}$. Let us define
\begin{equation*}
\gamma(s) = \begin{cases}
\eta(s) &\text{if}\; s\in [0,t],\\
\eta(t) &\text{if}\; s\in [t,\infty),
\end{cases}
\end{equation*}
so that $\gamma\in \mathcal{A}_{x}^k$.
Using Lemma \ref{DPP_lemma}, we have
\begin{align*}
u(x) &= \int_0^{t} e^{-s}L\left({\eta}(s),-\dot{\eta}(s)\right)\;ds + e^{-t}u(\eta(t)) \\
&\geq  \int_0^{t} e^{-s}L\left({\gamma}(s),-\dot{\gamma}(s)\right)\;ds - C_H e^{-t} \\
&\geq  \int_0^\infty e^{-s}L\left({\gamma}(s),-\dot{\gamma}(s)\right)\;ds - C_1 e^{-t}- C_He^{-t}\\
&\geq u_k(x) - \left((C_1+C_H) e^{\frac{|x|}{C_5}}\right) e^{-\frac{k}{C_5}}.
\end{align*}
Consequently, we obtain \eqref{exp:rate}. The conclusion for $|x|\leq R$ follows immediately.
\end{proof}

\begin{rem}\label{rem: bounded speed}
Here, we note that the constants in the proof above do not depend on the regularity of the Lagrangian. As long as a minimizer exists, we get the same exponential rate of convergence. See Appendix for a discussion on the existence of minimizers.
It is worth noting here that, for each $x\in \overline{B(0,k)}$, the existence of a minimizer  $\eta\in \mathcal{A}_{x}$  to  \eqref{rep_for} with bounded velocity is a nontrivial fact and plays an essential role in the proof above.
Moreover, the bound on the velocity of $\eta$ only depends on $C_H, A, B$.
\end{rem}

In the rest of this section, we provide two explicit examples to show that the rate $\mathcal{O}\left(e^{-\frac{k}{C}}\right)$ is indeed optimal.

\subsection{Examples with exponential rate of convergence}

\begin{prop}\label{example1} Let $H(p):\mathbb{R}\rightarrow\mathbb{R}$ be defined by $H(p) = |p-1|-1$ for $p\in [0,2]$ and $H(p)\geq 0$ elsewhere such that $H$ is continuous and coercive. Let $u_k$ be the solution to \eqref{u_k} on $[-k,k]$, then $u_k\rightarrow 0$ locally uniformly on $\mathbb{R}$ as $k\rightarrow\infty$. Here, $u\equiv 0$ is the unique solution to \eqref{u_eqn}. Furthermore, we have $u_k(k)=1$ for all $k\in \mathbb{N}$ and
\begin{equation*}
u_k(x) \geq e^{-2k} \;\text{on}\;[-k,k].
\end{equation*}
\end{prop}

\begin{proof} It is clear that $u_k(x) = e^{x-k}$ solves $v(x)+H(v'(x)) = 0$ in $(-k,k)$ in the classical sense, and indeed, in viscosity sense. We need to verify that $u_k$ is a viscosity supersolution on $[-k,k]$. Let $u_k-\varphi$ has a local minimum at $x=-k$ for $\varphi(x)\in \mathrm{C}^1 (\R)$. Clearly, we can see that 
\begin{equation*}
\varphi'(-k) \leq u_k'(-k) = e^{-2k},
\end{equation*}
which implies
$e^{-2k} + H(\varphi'(-k)) \geq 0$. On the other hand, at $x=k$, one has
\begin{equation*}
u_k(k)+H(\varphi'(k)) =1+H(\varphi'(k)) \geq 0
\end{equation*}
since by definition of $H$, it is bounded below by $-1$. Therefore, $u_k(x) = e^{x-k}$ is the unique viscosity solution to \eqref{u_k}, and furthermore $e^{-2k}\leq u_k(x) \leq \left(e^{|x|}\right)e^{-k}$ for all $x\in [-k,k]$. In addition to that, we have $u_k(k) = 1$ for all $k \in \N$, hence, the convergence fails when $x = k$.
\end{proof}

\subsection{Optimal control formulations}

We give another example  from the optimal control theory point of view (see \cite{Soner1986}). 
Let us recall briefly the setting of optimal control as follows. Let $U$ be a compact metric space. 
We regard a control as a Borel measurable map $\alpha:[0,\infty)\mapsto U$. 
Let $\Omega$ be an open subset of $\mathbb{R}^n$ with the connected boundary satisfying \eqref{A1}.
 We also assume that $b = b(x,a): \overline{\Omega} \times U\rightarrow \mathbb{R}^n$, $f = f(x,a):\overline{\Omega}\times U\rightarrow \mathbb{R}$ satisfy 
\begin{align*}
&\sup_{a\in U} \left|b(x,a) - b(y,a)\right| \leq L(b)|x-y| & &\text{for all}\;x,y\in \overline{\Omega},\\
&\sup_{a\in U} \left|b(x,a)\right|\leq K(b)& &\text{for all}\;x\in \overline{\Omega},\\
&\sup_{a\in U} \left|f(x,a) - f(y,a)\right| \leq \omega_f(|x-y|)& &\text{for all}\;x,y \in \overline{\Omega}, \\
&\sup_{a\in U} \left|f(x,a)\right| \leq K(f)& &\text{for all}\;x \in \overline{\Omega}, 
\end{align*}
where $K(b), L(b),K(f)$ are positive constants and $\omega_f$ is a nondecreasing continuous function with $\omega_f(0^+) = 0$.
\medskip

For each $x\in \overline{\Omega}$ and a given control  $\alpha(\cdot):[0,\infty)\rightarrow U$, let $y^{x,\alpha}(t)$ be a controlled process (we will write $\alpha$ instead of $\alpha(\cdot)$ as a control for simplicity), which is a solution to
\begin{equation*}
\begin{cases}
\frac{d}{dt} y^{x,\alpha}(t) = b\left(y^{x,\alpha}(t),\alpha(t)\right) \qquad \text { for } t >0,\\
\;\;\; y^{x,\alpha}(0) = x.
\end{cases}
\end{equation*}
We denote the set of controls (strategies) $\al$ where $y^{x,\alpha}(t)\in \overline{\Omega}$ for all $t\geq 0$ and $y^{x,\alpha}$ solves the ODE above by $\mathcal{A}_x$. The value function is defined by
\begin{equation*}
u(x) = \inf_{\alpha \in \mathcal{A}_x} \int_0^\infty e^{-t} f\big(y^{x,\alpha}(t),\alpha(t)\big)\;dt.
\end{equation*}
Here, one can define the Hamiltonian associated with $b$ and $f$ as
\begin{equation*}
H(x,p) := \sup_{a\in U} \left\lbrace -b(x,a)\cdot p - f(x,a)\right\rbrace \in \mathrm{C}\left(\overline{\Omega}\times \mathbb{R}^n;\mathbb{R}\right).
\end{equation*}
It was proved in \cite{Soner1986} that $u$ is a viscosity solution to \eqref{state-def}.

\begin{prop}\label{ex33} Let $n=1$ and $U = [-1,1]$. Let us consider the following Hamiltonian defined as
\begin{equation*}
\displaystyle H(x,p) = \sup_{a \in [-1,1]} \Big\lbrace -a\cdot p -e^{-|x|} \Big\rbrace = |p| - e^{-|x|}, \qquad (x,p)\in \mathbb{R}\times\mathbb{R}.
\end{equation*}
Then, the solution to \eqref{u_k} is given by $u_k(x) = \frac{e^{-|x|}}{2} + \frac{e^{|x|-2k}}{2}$ for $x\in [-k,k]$, while the solution to \eqref{u_eqn} is $u(x) = \frac{e^{-|x|}}{2}$. Hence, the exponential rate of convergence is obtained.
\end{prop}

\begin{proof} In the optimal control setting, the Hamiltonian above is obtained by considering $U=[-1,1]$, $b(x,a) = a$ and $f(x,a) = e^{-|x|}$. To find $u_k(x_0)$ and $u(x_0)$, one needs to find a control $\alpha(t)$ that minimizes
\begin{equation*}
\int_0^\infty e^{-s-|y(s)|}\;ds\qquad\text{subject to}\qquad \begin{cases}
\dot{y}(t) &= \alpha(t) \in [-1,1],\\
y(0) &= x_0.
\end{cases}
\end{equation*}
It is easy to see the following points:
\begin{itemize}
\item[$\mathrm{(i)}$] An optimal control for the unconstrained problem with $x_0\geq 0$ ($x_0<0$)  is $\alpha(t) \equiv 1$ ($\alpha(t) \equiv -1$, respectively).
\item[$\mathrm{(ii)}$] An optimal control for the constrained problem on $[-k,k]$ with $x_0\geq 0$ ($x_0<0$) is $\alpha(t) \equiv 1$ on $[0,k-x_0]$ and 0 elsewhere ($\alpha(t) \equiv -1$ on $[0,k+x_0]$ and 0 elsewhere, respectively).
\end{itemize}
Once we have the optimal controls, we can easily compute the value function and the result follows. In conclusion, for all $x\in [-k,k]$ we have
\begin{equation*}
0\leq u_k(x)-u(x)= \left(\frac{e^{|x|}}{2}\right)e^{-2k}.
\end{equation*}
In this example, the convergence holds everywhere in $[-k,k]$ with the rate $\mathcal{O}\left(e^{-k}\right)$.
\end{proof}

\begin{rem} 
One interesting fact to point out here is that the optimal path starting from $x_0$ for the state-constraint problem on $[-k,k]$ in Proposition \ref{ex33}
stays on the boundary $\pm k$ for all $t \geq  k - |x_0|$.
\end{rem}

\section{The case of bounded domain}\label{P2-bdd}
The second prototype case is considered in this section. 
Let us assume that (P2), $\eqref{H0}, \eqref{H1}$, $\eqref{H3c}$ and \eqref{H4} are enforced. 
Recall that $\Omega_k = B\left(0,1-\frac{1}{k}\right)$ for $k \in \N$, and $\Omega = B(0,1)$. 
Let $u_k\in \mathrm{Lip}\left(\overline{\Omega}_k\right)$ be the unique viscosity solution to
\begin{equation}\label{u_k-bdd}
\begin{cases}
u_k(x) + H(x,Du_k(x)) &\leq 0 \qquad\text{in}\;\Omega_k,\\
u_k(x) + H(x,Du_k(x)) &\geq 0 \qquad\text{on}\;\overline{\Omega}_k.
\end{cases}
\end{equation}
It is clear that we still have the following priori estimate
\begin{equation}\label{u_k-bdd-estimate}
\Vert u_k\Vert_{L^\infty(\overline{\Omega}_k)} + \Vert Du_k\Vert_{L^\infty(\overline{\Omega}_k)} \leq C_H.
\end{equation}

\begin{prop}\label{P2-conv} For each $k\in \mathbb{N}$, let $u_k$ be the unique solution to \eqref{u_k-bdd}. Then, there exists $u \in BUC (\overline{\Omega})$ such that $u_k\rightarrow u$ locally  uniformly on $\overline{\Omega}$ as $k$ grows to infinity. Moreover, $u$ has the same bounds as in \eqref{u_k-bdd-estimate} and solves
\begin{equation}\label{u-bdd}
\begin{cases}
u(x) + H(x,Du(x)) &\leq 0 \qquad\text{in}\; \Omega,\\
u(x) + H(x,Du(x)) &\geq 0 \qquad\text{on}\; \overline{\Omega}
\end{cases}
\end{equation}
in viscosity sense.
\end{prop}

\begin{proof} From a priori estimate \eqref{u_k-bdd-estimate}, by Arzel\`a-Ascoli's theorem and a diagonal argument we can extract a subsequence such that $u_{k_m}\rightarrow u$ uniformly on compact subsets of $\Omega$. By the stability of viscosity solutions we obtain that $u\in \mathrm{C}(\Omega)$ is a viscosity solution to
\begin{equation}\label{bdd_B(0,1)}
u(x) + H(x,Du(x)) = 0\qquad\text{in}\; \Omega.
\end{equation}
We deduce that $|u(x)|\leq C_B$ and $|u(x)-u(y)| \leq C_H|x-y|$ for $x,y\in \Omega$. We can extend $u\in \mathrm{Lip}(\overline{\Omega})$ with the same priori bound as in \eqref{u_k-bdd-estimate}. We need to show that $u$ is a viscosity supersolution to $u(x) + H(x,Du(x)) =  0$ on $\overline{\Omega}$. 

We can verify it using Corollary \ref{cor:Inv_max}. Indeed, let $v\in \mathrm{C}(\overline{\Omega})$ be a viscosity subsolution to \eqref{bdd_B(0,1)} in $\Omega$. Applying the comparison principle to $u_k(x) + H(x,Du_k(x)) \geq 0$ on $\overline{\Omega}_k$, we have that $v(x) \leq u_k(x)$ for $x\in \overline{\Omega}_k$. Now fixing $r\in (0,1)$, we have $v(x) \leq u_k(x)$ for all $x\in \overline{B(0,r)}$ and $r \leq 1-\frac{1}{k}$ if $k$ is large enough. Letting $k\rightarrow \infty$, we deduce that $v(x)\leq u(x)$ for $x\in \overline{B(0,r)}$. Since we have $u,v\in \mathrm{C}(\overline{B(0,1)})$, the inequality $v\leq u$ on $\overline{B(0,1)}$ follows. Hence, $u$ is a viscosity supersolution to \eqref{bdd_B(0,1)} by Corollary \ref{cor:Inv_max}.
\end{proof}

Now we are ready to give a proof for Theorem \ref{Conv-bdd2}. 
We note that star-shaped and scaling properties of $\{\Omega_k\}$ play an important role.

\begin{proof}[Proof of Theorem \ref{Conv-bdd2}] The fact that $u_k(x)\geq u(x)$ on $\overline{\Omega}_k$ is clear by the comparison principle. For $k\geq 2$, let us define 
\begin{equation*}
\tilde{u}_k(x):= \frac{k}{k-1}u_k\left(\frac{k-1}{k}x\right)\qquad \text{for}\;x\in \overline{B(0,1)}.
\end{equation*}
It is clear that $\tilde{u}_k$ is a viscosity subsolution to 
\begin{equation}\label{eq:u_k_scale_1}
\frac{k-1}{k}\tilde{u}_k(x) + H\left(\frac{k-1}{k}x, D\tilde{u}_k(x)\right) = 0 \qquad\text{in }  B(0,1).
\end{equation}
From \eqref{u_k-bdd-estimate} and \eqref{H3c}, there exists $\tilde{C}_H$ such that $|H(x,p) - H(x,p)|\leq \tilde{C}_H|x-y|$ for all $x,y\in \overline{\Omega}$ and $|p| \leq C_H$. Therefore, by using \eqref{eq:u_k_scale_1} we have
\begin{align*}
\tilde{u}_k(x) + H\left(x,D\tilde{u}_k(x)\right) &\leq \frac{1}{k}\tilde{u}_k(x) +H\left(x,D\tilde{u}_k(x)\right) -  H\left(\frac{k-1}{k}x,D\tilde{u}_k(x)\right)\leq \frac{C_H+\tilde{C}_H}{k}
\end{align*}
for all $x\in B(0,1)$. By the comparison principle and the fact that $u$ solves \eqref{u-bdd} in the viscosity sense, we deduce that
\begin{equation*}
\tilde{u}_k(x)-\frac{C_H+\tilde{C}_H}{k} \leq u(x) \qquad\text{for all}\; x\in \overline{B(0,1)}.
\end{equation*}
Consequently, we obtain the conclusion $u_k(x) \leq u(x) + \frac{C}{k}$ for $x\in \overline{\Omega}_k$ where the constant $C$ can be chosen as $C = 2C_H+\tilde{C}_H$.
%\begin{equation*}
%u_k\left(\frac{k-1}{k}x\right) \leq u(x) + \frac{C_H+\tilde{C}_H}{k} \leq u\left(\frac{k-1}{k}x\right)  + \frac{2C_H+\tilde{C}_H}{k}.
%\end{equation*}
%The conclusion
%\begin{equation*}
%u_k(x) \leq u(x) + \frac{C}{k}
%\end{equation*}
%for all $x\in \overline{\Omega}_k$ follows immediately where $C = 2C_H+\tilde{C}_H$.
\end{proof}

\begin{rem}\label{rem:P2} \quad
\begin{itemize}
\item[(i)] It is clear from the proof that prototype condition (P2) can be relaxed as following. \medskip
\begin{itemize}
\item[(P2')] Assume $0 \in \Omega_k$ for all $k \in \N$, $\Omega = \bigcup_{k \in \N} \Omega_k$ is bounded, and the comparison principle for the state-constraint problem holds on $\Omega_k,\Omega$.
Assume further that, for $k \in \N$,
\[
\left(1-\frac{1}{k}\right)\Omega \subset \Omega_k.
\]
\end{itemize}
\item[(ii)] Theorem \ref{Conv-bdd2} can also be proved using the doubling variable method with the following auxiliary function (see \cite{Capuzzo-Dolcetta1990})
\begin{equation*}
\Phi^{k}(x,y): = \tfrac{k+1}{k-1} u_k\left(\tfrac{k-1}{k+1}x\right) - u(y) - C_H k^{2}|x-y|^2
\end{equation*}
for $(x,y)\in \left(1+\tfrac{1}{k}\right)\overline{\Omega}\times \overline{\Omega}$.
\end{itemize}
\end{rem}

The following remark shows that the rate $\mathcal{O}\left(\frac{1}{k}\right)$ is indeed optimal.

\begin{rem} \label{rem: optimal for bounded domain}
Let $H$ be defined as in Proposition \ref{example1}, we see that $u_k(x) = e^{x-\left(1-\frac{1}{k}\right)}$ solves \eqref{u_k-bdd} and $u(x) = e^{x-1}$ solves \eqref{u-bdd}, therefore
\begin{equation*}
0\leq u_k(x) -u(x) =e^{x-1}\left(e^{\frac{1}{k}}-1\right)\leq \frac{2}{k}
\end{equation*}
for $x\in \left[-\left(1-\frac{1}{k}\right),1-\frac{1}{k}\right]$. Besides, $e^{\frac{1}{k}} - 1 \geq \frac{1}{k}$, and so, $\mathcal{O}\left(\frac{1}{k}\right)$ is optimal.
\end{rem}

\section{Discussions}\label{Sec:discussion}
We give here some further discussions along the line with the topics considered in the paper.
Firstly, when our Hamiltonian is given as $H(x,p)=a(x)K(p)$ in the first prototype (P1), we get an exponential rate of convergence provided that the assumption \eqref{H0} is enforced (Theorem \ref{npo}). 
Without this assumption, we have an example with a polynomial rate of convergence whose power can be increased or decreased as much as we want. 

\begin{ex} Let us consider $n=1$, $H(x,p) = \left(\frac{1+|x|}{m}\right)K(p)$ for $m>1$ and $K:\mathbb{R}\rightarrow\mathbb{R}$ defined by 
\begin{equation}\label{eqn:ex5}
K(p)= \begin{cases}
-|p| & \text{for}\; |p|\leq 1,\\
|p| - 2 & \text{for}\; |p|\geq 1.
\end{cases}
\end{equation}
The unique viscosity solution to \eqref{u_k} is
\begin{equation*}
u_k(x) = \frac{(1+|x|)^m}{m(1+k)^{m-1}} \qquad\text{for}\; x\in [-k,k].
\end{equation*}
Clearly, $u_k(x)\rightarrow 0$ locally uniformly with rate $\mathcal{O}\left(\frac{1}{k^{{m-1}}}\right)$ for any given $m>1$. 
We should note that the limit $0$ is not a unique solution to \eqref{u_eqn}. 
Another solution to  \eqref{u_eqn} is $u(x) =m^{-1}(1+|x|)^m$, but it does not belong to $\BUC(\R)$.
\end{ex}

\begin{ex}\label{ex22} 
Assume $n=1$, $H(x,p)=K(p) + V(x)$ where $V(x) = e^{-|x|}$ and $K:\mathbb{R}\rightarrow\mathbb{R}$ defined by
\begin{equation*}
K(p)= \begin{cases}
-|p| &\quad \text{for}\; |p|\leq 1,\\
|p|-2&\quad \text{for}\; |p|\geq 1.
\end{cases}
\end{equation*}
The unique state-constraint viscosity solution to \eqref{u_k} is
\begin{equation*}
\displaystyle u_k(x) = -\frac{1}{2}e^{-|x|} + \left(e^{-k}-\frac{1}{2}e^{-2k}\right)e^{|x|}, \qquad x\in [-k,k],
\end{equation*}
and the unique viscosity solution to \eqref{u_eqn} is
\begin{equation*}
u(x) = -\frac{1}{2}e^{-|x|}, \qquad x\in \mathbb{R}.
\end{equation*}
We have $u_k\rightarrow u$ locally uniformly in $\mathbb{R}$ with rate $\mathcal{O}(e^{-k})$.
\end{ex}

%\color{red}{(Rewrite) Add a generalization to this case.}
%\color{black}

%\begin{ex} \color{red} Examples with concave Hamiltonians + remarks on concave state-constraint Hamilton-Jacobi equations. \color{black}
%\end{ex}

%\color{red} \textbf{Some more examples }

%\color{black}

Secondly, prototype condition (P1) can be relaxed as follows.
\begin{rem}\label{rem:P1} \quad
It is clear from the proofs of our main results (Theorems \ref{thm:approx1/k^2}, \ref{npo}, \ref{thm:exp-rate}, and  Corollary \ref{cor1}) that prototype condition (P1) can be relaxed as following. \medskip
\begin{itemize}
\item[(P1')] Assume $\Omega_k$ is bounded, $B(0,k) \subset \Omega_k$, and the comparison principle for the state-constraint problem holds on $\Omega_k$ for all $k \in \N$.
Of course, $\Omega = \bigcup_{k \in \N} \Omega_k = \R^n$ here.
\end{itemize}

\end{rem}

Thirdly, there are some open questions we are not able to answer yet.

\begin{quest} \label{Q1}
In the first prototype {\rm(P1)} case, what is the optimal rate of convergence of $u_k$ to $u$ in the general nonconvex setting?
\end{quest}

\noindent A more specific question is as follows.

\begin{quest}\label{Q2}
Assume {\rm(P1)}, and $H(x,p) = K(p)+V(x)$, where $K \in \Lip(\R^n)$ is coercive and nonconvex, and $V\in\mathrm{BUC}(\mathbb{R}^n)$.
Is it true that we always have an exponential rate of convergence of $u_k$ to $u$?
\end{quest}

\section{Appendix} \label{appendix}

\subsection{Proofs of some lemmas}

\begin{proof}[Proof of Lemma \ref{growth-L'-lem}] We first prove the result for all $(x,v)$ with $|v|\leq 1$, then by scaling we get the result for all $(x,v)$. Using \eqref{L7} we have $-B \leq L(x,v) \leq 4A+B$ for all $(x,v)$ with $|v|\leq 2$. For $u,v\in B(0,1)$ with $u\neq v$, let $w = v + |v-u|^{-1}(v-u)$. Then, $|w|< |v| + 1 < 2$. Let $\lambda = (1+|u-v|)^{-1}\in (0,1)$, we have $v = \lambda u + (1-\lambda)w$. By the convexity, one obtains
\begin{align*}
L(x,v) - L(x,u) &\leq (1-\lambda)\big(L(x,w) - L(x,u)\big) \leq (4A+2B)|u-v|.
\end{align*}
By symmetry, we deduce that $|L(x,u) - L(x,v)| \leq (4A+2B)|u-v|$ for all $(x,v)$ with $|v|\leq 1$. In other words, we have that $|\xi| \leq 4A+2B$ whenever $\xi \in D^-_v L(x,v)$ for $(x,v)\in \mathbb{R}^n\times \overline{B(0,1)}$. Now for $r>1$, we define $L_r(x,v) = r^{-2}L(x,rv)$ for $(x,v)\in \mathbb{R}^n \times \mathbb{R}^n$. We observe that
\begin{equation*}
A^{-1}|v|^2 - B\leq A^{-1}|v|^2-B r^{-2}\leq   L_r(x,v)
\leq A|v|^2 + B r^{-2}\leq A|v|^2+B
\end{equation*}
for all $(x,v)$. For $v\in \mathbb{R}^n$ with $|v|\geq 1$, let $r=2|v|>1$ and $u = \frac{v}{2|v|} \in B(0,1)$ so that $v = ru$. Since $\xi\in D_v^-L(x,v)$ implies $\frac{\xi}{r}\in D_u^-L_r\left(x,u\right)$, we have $\left|\xi\right| \leq (4A+2B)|r| = (8A+4B)|v|$.
\end{proof}

\begin{proof}[Proof of Theorem \ref{Existence of minimizer}] Let $\{\eta_k\}_{k=1}^\infty \subset\mathcal{A}_{x}$ be a minimizing sequence in $\mathrm{AC}([0,\infty))$ such that $\lim_{k\rightarrow\infty} J\left[x,\eta_k\right] = u(x)$. From the uniform boundedness of $u$ and the quadratic bounds of $L(x,v)$, we have
\begin{equation*}
\left\Vert e^{-\frac{s}{2}}\dot{\eta}_k(s) \right\Vert_{L^2\left((0,\infty);\mathbb{R}^n\right)} \leq C_4.
\end{equation*}
Here, $C_4$ can be chosen as $(A(2C_H+B))^\frac{1}{2}$. By the weak compactness of $L^2$, there exists $g$ such that $e^{-\frac{s}{2}}g(s) \in L^2\big((0,\infty);\mathbb{R}^n\big)$ and a subsequence $\{k_j\}$ $\rightarrow \infty$ such that $e^{-\frac{s}{2}}\dot{\eta}_{k_j} \rightharpoonup e^{-\frac{s}{2}}g$ weakly in $L^2\big((0,\infty);\mathbb{R}^n\big)$ as $j\rightarrow\infty$. 
\medskip

Writing $g$ as $e^{\frac{s}{2}}g \cdot e^{-\frac{s}{2}}$ and using the Cauchy-Schwartz inequality, we get $g\in L^1_{\mathrm{loc}}\big((0,\infty);\mathbb{R}^n\big)$. For $t>0$, we let $\eta(t) = x+\int_0^t g(s)\;ds$. Clearly, $\eta \in \mathcal{A}_{x}$ and one obtains that $\eta_{k_j}\rightarrow \eta$ pointwise with $\dot{\eta}=g$ almost everywhere. On the other hand, the convexity of $L$ implies
\begin{equation*}
L\left(\eta_{k_j}(s),-\dot{\eta}_{k_j}(s)\right) \geq L\left(\eta_{k_j}(s),-\dot{\eta}(s)\right) - D_vL\left(\eta_{k_j}(s),-\dot{\eta}(s)\right)\cdot \left(\dot{\eta}_{k_j}(s) - \dot{\eta}(s)\right).
\end{equation*}
Therefore, 
\begin{align*}
\int_0^\infty e^{-s} L\left(\eta_{k_j}(s),-\dot{\eta}_{k_j}(s)\right)ds &\geq \int_0^\infty e^{-s} L\left(\eta_{k_j}(s),-\dot{\eta}(s)\right)ds \nonumber\\
&\qquad  + \int_0^\infty e^{-s/2}D_vL\left(\eta_{k_j}(s),-\dot{\eta}(s)\right)\cdot e^{-s/2}\left(\dot{\eta}_{k_j}(s) - \dot{\eta}(s)\right)ds. 
\end{align*}
Since $|D_vL\left(\eta_{k_j}(s),-\dot{\eta}(s)\right)| \leq C_L(1+|\dot{\eta}(s)|)$ for a.e. $s\in (0,\infty)$, it is clear that 
\begin{equation*}
e^{-s/2}D_vL\left(\eta_{k_j}(s),-\dot{\eta}(s)\right) \rightarrow e^{-s/2}D_vL\left(\eta(s),-\dot{\eta}(s)\right)
\end{equation*}
in $L^2((0,\infty);\mathbb{R}^n)$ and thus
\begin{equation*}
\int_0^\infty e^{-s}D_vL\left(\eta_{k_j}(s),-\dot{\eta}(s)\right)\cdot\left(\dot{\eta}_{k_j}(s) - \dot{\eta}(s)\right)ds
\end{equation*}
converges to 0 as $k$ goes to infinity, which yields that $J[x,\eta] \leq u(x)$. Hence $J[x,\eta] = u(x)$.
\end{proof}

\begin{proof}[Proof of Lemma \ref{DPP_lemma}] By the definition of $u$ in \eqref{def.u}, we have 
\begin{equation*}
u\big(\eta(t)\big) \leq \int_0^\infty  e^{-s} L\big(\gamma(s),-\dot{\gamma}(s)\big)\;ds = e^t \int_t^\infty e^{-\xi}L\left(\eta(\xi),-\dot{\eta}(\xi)\right)\;d\xi,
\end{equation*}
where $\gamma(s) = \eta(t+s)$ for $s\geq 0$. Thus,
\begin{equation}\label{char_3}
e^{-t}u\big(\eta(t)\big) \leq \int_t^\infty e^{-\xi}L\left(\eta(\xi),-\dot{\eta}(\xi)\right)\;d\xi.
\end{equation}
By the dynamic programming principle and \eqref{char_3}, we have
\begin{align*}
u\big(\eta(0)\big) &\leq \int_0^t  e^{-s}L \big(\eta(s),-\dot{\eta}(s)\big)\;ds  + e^{-t}u\big(\eta(t)\big) \leq \int_0^\infty  e^{-s} L\big(\eta(s),-\dot{\eta}(s)\big)\;ds =  u\big(\eta(0)\big).
\end{align*}
Therefore, \eqref{DPP_lem}, \eqref{char_4} and \eqref{char_5} follow.
\end{proof}

\begin{proof}[Proof of Lemma \ref{bound_on_doteta}] For every $t,h>0$, by Lemma \ref{DPP_lemma} we have that
\begin{equation*}
\frac{e^{-t}u\big(\eta(t)\big) - e^{-(t+h)}u\big(\eta(t+h)\big)}{h}  = \frac{1}{h}\int_t^{t+h}  e^{-s} L\big(\eta(s),-\dot{\eta}(s)\big)\;ds .
\end{equation*}
Let $\varphi\in \mathrm{C}^1(\mathbb{R})$ such that $u-\varphi$ has a local min at $\eta(t)$ and $u\big(\eta(t)\big) = \varphi\big(\eta(t)\big)$, then
\begin{equation*}
\frac{e^{-t}u\big(\eta(t)\big) - e^{-(t+h)}u\big(\eta(t+h)\big)}{h} \leq \frac{e^{-t}\varphi\big(\eta(t)\big) - e^{-(t+h)}\varphi\big(\eta(t+h)\big)}{h}.
\end{equation*}
Therefore,
\begin{equation*}
\frac{1}{h}\int_t^{t+h}  e^{-s} L\Big(\eta(s),-\dot{\eta}(s)\Big)\;ds  \leq \frac{e^{-t}\varphi\big(\eta(t)\big) - e^{-(t+h)}\varphi\big(\eta(t+h)\big)}{h}.
\end{equation*}
Since $\eta(t)$ is differentiable a.e. in $(0,\infty)$, at those $t$ where $\eta(t)$ is differentiable, let $h\rightarrow 0^+$ we deduce that
\begin{equation*}
e^{-t}L\big(\eta(t),-\dot{\eta}(t)\big) \leq -\frac{d}{dt}\Big(e^{-t}\varphi\big(\eta(t)\big)\Big) = e^{-t}\varphi\big(\eta(t)\big) - e^{-t}D\varphi\big(\eta(t)\big)\cdot\dot{\eta}(t).
\end{equation*}
Thus, for a.e. $t>0$ where $\eta$ is differentiable, we have
\begin{equation*}
L\big(\eta(t),-\dot{\eta}(t)\big) \leq \varphi\big(\eta(t)\big) - D\varphi\big(\eta(t)\big)\cdot\dot{\eta}(t).
\end{equation*}
By \eqref{L7} and  a priori estimate \eqref{est.u} for a.e. $t\in (0,\infty)$ we have that 
\begin{equation*}
A^{-1}|\dot{\eta}(t)|^2 - B \leq \varphi\big(\eta(t)\big) - D\varphi\big(\eta(t)\big)\cdot\dot{\eta}(t) \leq C_H+C_H|\dot{\eta}(t)|.
\end{equation*}
This shows that $|\dot{\eta}(t)|\leq C_5$ for a.e. $t\in (0,\infty)$, and $C_5$ only depends on $C_H,A,B$. 
\end{proof}

It is worth emphasizing again here that the bound $C_5$ on the velocity of $\eta$ only depends on $C_H, A, B$, which can be seen clearly from the last chain of inequalities in the above proof.
In fact, one can choose explicitly that $C_5 = (2AB+2AC_H+A^2C_H^2)^{1/2}$.

\subsection{Existence of minimizers in the general case}
We show that one can remove the smoothness of $L$ in Theorem \ref{Existence of minimizer} under the assumption $\mathrm{(L3)}$.

\medskip
Let us consider mollifiers in $\mathbb{R}^{2n}$ defined as $\{\eta_{\varepsilon}\}_{\varepsilon>0}$ such that $\eta_\varepsilon(x) = \frac{1}{\varepsilon^{2n}}\eta\left(\frac{x}{\varepsilon}\right)$ for $x\in \mathbb{R}^{2n}$ where $\eta\in \mathrm{C}_c^\infty(\mathbb{R}^{2n})$ satisfying $0\leq \eta \leq 1$, $\mathrm{supp}\;(\eta) \subset B_{\mathbb{R}^{2n}}(0,1)$ and $\int_{\mathbb{R}^{2n}} \eta(x)\;dx = 1$.

For each $\varepsilon>0$ we define the convolution $L^\varepsilon = \eta_\varepsilon*L \in \mathrm{C}^\infty(\mathbb{R}^n\times\mathbb{R}^n)$. It is easy to see that $L^\varepsilon$ is bounded below, $\mathrm{(L5)}, \mathrm{(L6)}$ are preserved to $L^\varepsilon$ and $\mathrm{(L7)}$ now becomes:
\begin{itemize}
\item[(L7$^\varepsilon$)] There exist positive constants 
$A_\varepsilon, B_\varepsilon$ such that $A^{-1}_\varepsilon - B^{-1}_\varepsilon \leq L^\varepsilon(x,v) \leq A_\varepsilon|v|^2 + B_\varepsilon$ for all $(x,v)\in \mathbb{R}^n\times\mathbb{R}^n$, and $A_\varepsilon\rightarrow A, B_\varepsilon\rightarrow B$ as $\varepsilon\rightarrow 0$.
\end{itemize}
By Theorem \ref{Existence of minimizer}, there exists a minimizer $\gamma_\varepsilon$ in $\mathcal{A}_x$ such that 
\begin{equation*}
u^\varepsilon(x):= \inf_{\zeta\in \mathcal{A}_x} \int_0^\infty e^{-s}L^\varepsilon\big(\zeta(s),-\dot{\zeta}(s)\big)\;ds  = \int_0^\infty e^{-s}L^\varepsilon\big(\gamma_\varepsilon(s),-\dot{\gamma}_\varepsilon(s)\big)\;ds.
\end{equation*}
Let $H^\varepsilon$ be the Legendre transform of $L^\varepsilon$.
Then, we can show that $u^\varepsilon$ is the unique solution to $u^\varepsilon(x)+H^\varepsilon(x,Du^\varepsilon(x)) = 0$ in $\mathbb{R}^n$.  It is easy to see that $H^\varepsilon\rightarrow H$ locally uniformly in $\mathbb{R}^n\times\mathbb{R}^n$, therefore by stability of viscosity solutions, $u^\varepsilon\rightarrow u$ locally uniformly in $\mathbb{R}^n$ as $\varepsilon\rightarrow 0$.
\medskip

We indeed have that $\gamma^\varepsilon$ is smooth according to Remark \ref{remark:Euler-Lagrange}. Furthermore, Theorem \ref{Existence of minimizer} yields that $\Vert e^{-\frac{s}{2}}\dot{\gamma}_\varepsilon(s)\Vert_{L^2} \leq C$ and $|\dot{\gamma}_\varepsilon| \leq C$ pointwise in $(0,\infty)$. Therefore, we can define $\gamma\in \mathcal{A}_x$ such that (up to subsequence) $\gamma_\varepsilon\rightarrow \gamma$ locally uniformly on $[0,\infty)$ and $e^{-\frac{s}{2}}\dot{\gamma}_\varepsilon \rightharpoonup e^{-\frac{s}{2}}\dot{\gamma}$ weakly in $L^2$. Since $L^\varepsilon\rightarrow L$ uniformly on a compact set and $\left\lbrace\ -\dot{\gamma}_\varepsilon(s)\right\rbrace _{\varepsilon>0}$ is bounded, we obtain that
\begin{equation*}
L^\varepsilon\big(\gamma_\varepsilon(s),-\dot{\gamma}_\varepsilon(s)\big) = L\big(\gamma(s),-\dot{\gamma}_\varepsilon(s)\big) + \tilde{\omega}_{C_5}(\varepsilon)+\tilde{\omega}_{C_5}(|\gamma_\varepsilon(s)-\gamma(s)|)
\end{equation*}
using $\mathrm{(L3)}$. Therefore, it suffices to show that
\begin{equation}\label{exis_app1}
\int_0^\infty e^{-s} L(\gamma(s),-\dot{\gamma}(s))\;ds \leq \liminf_{\varepsilon\rightarrow 0}\int_0^\infty e^{-s}L\big(\gamma(s),-\dot{\gamma}_\varepsilon(s)\big)\;ds.
\end{equation}
For simplicity, let $d\mu = e^{-s}ds$ be a probability measure on $[0,\infty)$. It is easy to see that the functional $I:L^2(\mu)\rightarrow \mathbb{R}$ maps $f\mapsto \int_0^\infty L(\gamma(s),f(s))d\mu(s)$ is convex and lower semicontinuous, thus it is also weakly lower semicontinuous. Now since $\dot{\gamma}_\varepsilon \rightharpoonup \dot{\gamma}$ weakly in $L^2(d\mu)$, we obtain \eqref{exis_app1} and thus $\gamma$ is a minimizer for $u(x)$.

\begin{rem} Inequality \eqref{exis_app1} for the  Cauchy problem (finite time horizon) is proved using a different argument by H. Ishii in \cite{Ishii2008} under more general assumptions. Such inequalities are crucial for the analysis of large time behavior of solutions to the time-dependent problems.
\end{rem}

%\nocite{*}

%\section*{Acknowledgement}

\bibliography{zzzzlibrary}{}
\bibliographystyle{acm}

\end{document}